\documentclass[12pt, oneside,reqno]{amsart}
\usepackage{amsmath, amsfonts, amssymb}
\usepackage{eucal}
\usepackage{mathrsfs}
\usepackage{latexsym}
\usepackage{bbm}
\usepackage{mathtools}
\usepackage{exscale}
\usepackage{cases}

\usepackage[pdfstartview=FitH,
            CJKbookmarks=true,
            bookmarksnumbered=true,
            bookmarksopen=true,
            colorlinks=true,
            linkcolor=blue,
            anchorcolor=blue,
            citecolor=red,
            urlcolor=blue
            ]{hyperref}

\numberwithin{equation}{section}
\newtheorem{theorem}{Theorem}[section]
\newtheorem{definition}[theorem]{Definition}

\newtheorem{corollary}[theorem]{Corollary}
\newtheorem{proposition}[theorem]{Proposition}

\newtheorem{definition and theorem}[theorem]{Definition and Theorem}

\def\bl{\begin{lem}}
\def\el{\end{lem}}
\def\bc{\begin{corollary}}
\def\ec{\end{corollary}}
\def\bt{\begin{thm}}
\def\et{\end{thm}}

\def\bp{\begin{proposition}}
\def\ep{\end{proposition}}
\def\be{\begin{equation}}
\def\ee{\end{equation}}
\def\baa{\begin{align*}}
\def\eaa{\end{align*}}

\def\bd{\begin{definition}}
\def\ed{\end{definition}}

\allowdisplaybreaks[4]
\theoremstyle{plain}
\newtheorem{thm}{Theorem}[section]
\newtheorem*{Thm1}{Theorem 1}
\newtheorem*{Thm2}{Theorem 2}

\newtheorem{lem}[thm]{Lemma}
\newtheorem{prop}[thm]{Proposition}
\newtheorem{cor}[thm]{Corollary}

\theoremstyle{remark}

\newcommand{\lsub}[1]{\hskip -1.0pt\lower.3ex\hbox{$_{#1}$}}

\theoremstyle{definition}

\newcommand{\sn}{\mathbb S^{n-1}}

\newcommand{\tr}{\mathbb R}
\newcommand{\rn}{\mathbb R^n}

\newcommand{\oc}{\gamma_{n,\alpha}}
\newcommand{\Sa}{{\rm S}_{\alpha}f}

\newcommand{\Snf}{{\rm S}_{n}f}
\newcommand{\di}{\int_{\rn}\int_{\rn}}
\newcommand{\rS}{{\rm S}}
\newcommand{\rR}{{\rm R}}
\newcommand{\Ra}{{\rm R}_{\alpha}}
\newcommand{\Rf}{{\rm R}_{\alpha} f}

\newcommand{\Pit}{\Pi_2^{*,-\alpha/2}f}
\newcommand{\Rmnum}[1]{\expandafter\@slowromancap\romannumeral #1@} % 大写
\renewcommand{\chi}{\operatorname{1}}

\title[Affine Logarithmic HLS and Beckner-Type Logarithmic Sobolev Inequalities]{Affine Logarithmic HLS and Beckner-Type Logarithmic Sobolev Inequalities}
\author{Xiaxing Cai}

\begin{document}

\begin{abstract}
    In this paper, we consider two limiting cases ($\alpha\to n$ and $\alpha\to0$) of the affine HLS inequalities introduced by Haddad and Ludwig. We establish an affine logarithmic HLS inequality and an affine version of Beckner’s logarithmic Sobolev inequality, both of which are stronger than their classical counterparts — namely, the logarithmic HLS inequality of Carlen and Loss (1992) and Beckner (1993), and Beckner’s logarithmic Sobolev inequality (1995). For both inequalities, we obtain sharp constants together with a complete characterization of all equality cases.
\end{abstract}

\maketitle

\section{Introduction}

Affine inequalities, that is, inequalities which remain invariant under translations and volume-preserving linear transforms, have attracted considerable attention in the study of functional inequalities in recent years. The best-known affine inequality is the affine Sobolev inequality by Zhang  \cite{Zh}, which strengthens the classical Sobolev inequality. It was extended to functions of bounded variation by Wang  \cite{Wa}. Affine versions of $L^p$ Sobolev inequalities were established by Lutwak, Yang and Zhang  \cite{LYZ1, LYZ2} and Haberl and Schuster  \cite{HS}. Recently, affine versions of fractional $L^p$ Sobolev inequalities were established by Haddad and Ludwig  \cite{HL1, HL2}. 

%Moreover, the fractional $L^2$ Sobolev inequalities are equivalent by duality to the Hardy-Littlewood-Sobolev inequalities (HLS inequalities) and see  \cite{Ca, LL} for more information.

More recently, Haddad and Ludwig  \cite{HL3} established the following affine Hardy-Littlewood-Sobolev inequalities (affine HLS inequalities) which are significantly stronger than the sharp HLS inequalities by Lieb  \cite{Lieb}.
\vskip 6pt

\noindent{\bf Affine HLS inequalities (Haddad and Ludwig  \cite{HL3}):} For $\alpha\in(0,n)$ and non-negative $f\in L^{2n/(n+\alpha)}(\rn)$,
\begin{equation}\label{AHLS}
\begin{split}
    \gamma_{n,\alpha}\|f\|^2_{\frac{2n}{n+\alpha}}\ge n\omega_n^{\frac{n-\alpha}{n}}\Big(\frac{1}{n}\int_{\sn}\Big(\int_0^{\infty}t^{\alpha-1}\int_{\rn}  f(x)f(x+tu)dxdt\Big)^{\frac{n}{\alpha}} \Big)^{\frac{\alpha}{n}}\\
    \ge \int_{\rn}\int_{\rn}\frac{f(x)f(y)}{|x-y|^{n-\alpha}}dxdy.
\end{split}
\end{equation}
 There is equality in the first inequality precisely if $f(x)=a(1+|\phi(x-x_0)|^2)^{-(n+\alpha)/2}$ for $x\in\rn$ with $a\ge 0$, $\phi\in {\rm GL}(n)$ and $x_0\in\rn$. There is equality in the second inequality if $f$ is radially symmetric.

Here $|\cdot|$ is the Euclidean norm in $\rn$ and $\omega_n$ denotes the volume of the $n$-dimensional unit ball. The function $f\in L^p(\rn)$ is a measurable function with $\|f\|_p^p:=\int_{\rn}|f(x)|^pdx<\infty$.
The optimal constant $\gamma_{n,\alpha}$ comes from the classic HLS inequalities and is given by
\begin{equation}\label{const 1}
    \gamma_{n,\alpha}=\pi^{\frac{n-\alpha}{2}}\frac{|\Gamma(\frac{\alpha}{2})|}{\Gamma(\frac{n+\alpha}{2})}\Big(\frac{\Gamma(n)}{\Gamma(\frac{n}{2})}\Big)^{\frac{\alpha}{n}},
\end{equation}
where $\Gamma$ is the gamma function. Haddad and Ludwig also showed that the inequality \eqref{AHLS} reverses for $\alpha>n$, strengthening the sharp HLS inequalities for $\alpha>n$ by Dou and Zhu  \cite{DZ} and Beckner  \cite{Be}.

Eliminating the intermediate term in \eqref{AHLS} gives the classic HLS inequalities. As a limiting case of HLS inequalities ($\alpha\rightarrow n$), Carlen and Loss  \cite{CL} established the following sharp logarithmic HLS inequality as well as the equality case.
\vskip 6pt

\noindent{\bf The logarithmic HLS inequality (Carlen and Loss  \cite{CL}):} For non-negative functions $f\in L^1(\rn)$ such that $\|f\|_1=1$, $\int_{\rn}f(x)\log(1+|x|^2)dx<\infty$ and $\int_{\rn}f(x)\log f(x)dx$ is finite,
\begin{equation}\label{log HLS}
-\int_{\rn}\int_{\rn}f(x)\log|x-y|f(y)dxdy\leq\frac{1}{n}\int_{\rn}f(x)\log f(x)dx+ \gamma_n.
\end{equation}
The constant $\gamma_n=\frac{d}{d\alpha}
\big|_{\alpha=n}\gamma
_{n,\alpha}$ is given by 
\begin{equation*}\label{gamma n}
\gamma_n=\frac{1}{2}\log\pi+\frac{1}{n}\log\frac{\Gamma(\frac{n}{2})}{\Gamma(n)}+\frac{1}{2}\big(\psi(n)-\psi(\frac{n}{2})\big),
\end{equation*}
where $\psi$ is the logarithmic derivative of the gamma function.  Equality holds in \eqref{log HLS} precisely if $f(x)=a(1+\lambda|x-x_0|^2)^{-n}$ for $x_0\in\rn$ and $a, \lambda>0$ such that $\|f\|_1=1$.

The aim of this paper is to establish the limiting cases of the affine HLS inequalities. We use the star-shaped set $\Sa$ introduced by Haddad and Ludwig in  \cite{HL3}. 
For a measurable function $f:\rn\rightarrow [0,\infty)$ and $\alpha>0$, the star-shaped set $\Sa$ related to HLS inequalities, is defined by its radial function as
\[\rho_{\Sa}(u)^{\alpha}=\int_0^{\infty}t^{\alpha-1}\int_{\rn}f(x)f(x+tu)dxdt,\]
for $u\in\sn$. 
The affine HLS inequalities \eqref{AHLS} for $\alpha\in(0,n)$ can be written as
\begin{equation}\label{aff HLS}
    \gamma_{n,\alpha}\|f\|_{\frac{2n}{n+\alpha}}^2\ge n\omega_n^{\frac{n-\alpha}{n}}|\Sa|^{\frac{\alpha}{n}}\ge\di\frac{f(x)f(y)}{|x-y|^{n-\alpha}}dxdy,
\end{equation}
where $|\cdot|$ denotes the $n$-dimensional Lebesgue measure. Since 
\[{\rm S}_{\alpha}(f\circ \phi^{-1})=\phi\Sa\]
for volume-preserving linear transformation $\phi:\rn\rightarrow\rn$, the first inequality from \eqref{aff HLS} is affine invariant.

%The limiting case of the affine HLS inequalities as $\alpha\rightarrow n$ is the following affine logarithmic HLS inequality, which is stronger than the Euclidean case by Carlen and Loss. In Section $3$, we prove Theorem 1 by applying Riesz's rearrangement inequality for the logarithmic kernel, which extends the result of Carlen and Loss \cite[Lemma 1]{CL}. While Theorem $1$ can also be obtained by differentiating each term in \eqref{aff HLS}, this approach does not allow for the characterization of equality cases, which the rearrangement method handles naturally.

The limiting case of the affine HLS inequalities as $\alpha\rightarrow n$ leads to the following affine logarithmic HLS inequality, which is stronger than the Euclidean case by Carlen and Loss. 

\begin{Thm1}[The affine Logarithmic HLS inequality]
    For non-negative $f\in L^1(\rn)$ such that $\|f\|_1=1$,  $\int_{\rn}f(x)\log(1+|x|^2)dx<\infty$ and $\int_{\rn}f(x)\log f(x)dx$ is finite,  
{ 
\begin{equation*}
    \begin{aligned}
        -\int_{\rn}&\int_{\rn}f(x)\log|x-y|f(y)dxdy\\
        &\leq-\int_{\rn}\int_{\rn}f(x)\log\|x-y\|_{\Snf}f(y)dxdy+\frac{1}{n}\log(n\omega_n)\leq\frac{1}{n}\int_{\rn}f(x)\log f(x)dx+\gamma_n.
    \end{aligned}
\end{equation*}
}There is equality in the second inequality if and only if $f(x)=a(1+|\phi(x-x_0)|^2)^{-n}$ for $x_0\in\rn$, $a>0$ and $\phi\in {\rm GL}(n)$ such that $\|f\|_1=1$. There is equality in the first inequality if $f$ is radially symmetric.
\end{Thm1}

Theorem 1 will be proved in Section 3 by applying Riesz's rearrangement inequality for the logarithmic kernel, which extends the result of Carlen and Loss \cite[Lemma 1]{CL}.
Although Theorem 1 can also be obtained by differentiating each term in \eqref{aff HLS}, the rearrangement approach has the additional advantage of facilitating a precise characterization of the equality cases.

The limiting behavior of the HLS inequalities as $\alpha\rightarrow 0^+$ was considered by Beckner  \cite{Be2}. Beckner shows that the HLS inequalities become an identity in the limit and establishes the following logarithmic Sobolev inequality by differentiation of the HLS inequalities.
\vskip 6pt

\noindent{\bf Beckner's logarithmic Sobolev inequality  \cite{Be2}}: Let $\beta\in(0,\frac{1}{2})$. For non-negative smooth rapidly decreasing functions $f\in W^{\beta,2}(\rn)$ such that $\|f\|_2=1$,
\[\int_{\rn}|\hat{f}(\xi)|^2\log|\xi|d\xi\ge \frac{2}{n}\int_{\rn}f(x)^2\log f(x)dx+\psi(\frac{n}{2})-\frac{1}{2}\log\pi-\frac{1}{n}\log\frac{\Gamma(n)}{\Gamma(\frac{n}{2})},\]
where $W^{\beta,2}(\rn)$ is the fractional Sobolev space defined by
\[W^{\beta,2}(\rn)=\Big\{f\in L^2(\rn):\di\frac{|f(x)-f(y)|^2}{|x-y|^{n+2\beta}}dxdy<\infty\Big\}\]
and $\hat f$ is the Fourier transform of $f$ given by
\[\hat{f}(\xi)=\int_{\rn}e^{-2\pi i\xi\cdot x}f(x)dx.\]
Moreover, the equality is attained { if and only if} $f(x)=a(1+\lambda|x-x_0|^2)^{-\frac{n}{2}}$ for $a,\lambda>0$ and $x_0\in\rn$ such that $\|f\|_2=1$.

In Section $4$, we establish the following affine version of Beckner's logarithmic Sobolev inequality, which is the limiting case of the affine HLS inequalities as $\alpha\rightarrow 0^+$.  In addition, this new affine inequality is also a limiting case of affine fractional $L^2$ Sobolev inequalities, which is introduced in Section $2$.

\begin{Thm2}
    For $\beta\in(0,\frac{1}{2})$, suppose { $f\in W^{\beta,2}(\rn)$} is non-negative with   $\|f\|_2=1$ and $\int_{\rn}f(x)^2\log f(x)dx$ is finite. Then
    %normalized version
    {\begin{align}\label{main 2i}
    \gamma_0-\frac{2}{n}\int_{\rn}f(x)^2\log f(x)dx
     \ge\frac{1}{n}\log\Big(\frac{|{\rm R}_0f|}{\omega_n}\Big)\ge\frac{1}{n\omega_n}\int_{\sn}\log\rho_{{\rm R}_0f}(u)du,
    \end{align}}
    %{common version without normalization}{\cb \begin{align*}
    %\gamma_0\|f\|_2^2+\omega_n\|f\|_2^2\log\|f\|_2^2-&2\omega_n\int_{\rn}f(x)^2\log f(x)dx\\
    % &\ge\omega_n\|f\|_2^2\log\left(\frac{|{\rm R}_0f|}{\omega_n}\right)\ge\|f\|_2^2\int_{\sn}\log\rho_{{\rm R}_0f}(u)du,
    %\end{align*}}
    where $\rR_0 f$ is the star-shaped set given by
    \[\log\rho_{\rR_0f}(u)=-\gamma+\int_0^{\infty}\frac{1}{t}\Big(\int_{\rn}f(x)f(x+tu)dx-e^{-t}\Big)dt.\]
    Here $-\gamma=\psi (1)$ is the Euler constant. The constant $\gamma_0=\frac{1}{n\omega_n}\frac{d}{dt}\big|_{\alpha=0}\alpha\gamma_{n,\alpha}$ is given by
    \[\gamma_0=-\frac{1}{2}\log\pi+\frac{1}{n}\log\frac{\Gamma(n)}{\Gamma(\frac{n}{2})}+\frac{1}{2}\big(\psi(1)-\psi(\frac{n}{2})\big).\]
    There is equality in the first inequality if and only if $f(x)=a(1+|\phi(x-x_0)|^2)^{-\frac{n}{2}}$ for $x_0\in\rn$, $a> 0$ and $\phi\in {\rm GL}(n)$ such that $\|f\|_2=1$. There is equality in the second inequality if $f$ is radially symmetric.
\end{Thm2}

The set $\rR_0f$ here is the $0$-th radial mean body and
\[\rho_{\rR_0 f}(u)=\lim_{\alpha\rightarrow 0^+}\rho_{\Ra f}(u),~~\text{for almost all}~~u\in\sn,\]
 where the $\alpha$-th radial mean body $\Ra f$ is defined by a certain normalization of $\Sa$. Hence for a volume-preserving transform $\phi:\rn \rightarrow \rn$, we have
\[\rR_0(f\circ\phi^{-1})=\phi\rR_0f,\]
which implies the affine invariance of the first inequality from \eqref{main 2i}. The precise definition of $\Ra f$ is given by Haddad and Ludwig \cite{HL3}, which will be recalled in Section $2$. 

Eliminating the intermediate term in \eqref{main 2i} gives Beckner's inequality in an equivalent form for non-negative and non-zero functions $f\in W^{\beta,2}(\rn)\cap L^1(\rn)$, as will be shown in Section 4. While Beckner's original paper only briefly touches on the equality conditions, Section 5 provides a comprehensive characterization for both Beckner's inequality and its affine counterpart. In particular, by combining a central result of that section, Theorem \ref{main 5}, with technical results from Carlen and Loss \cite{CL}, we give a clear and detailed explanation of the equality conditions for Beckner's inequality. Building on the same result, we then obtain a complete description of the equality cases for the affine version.
    
\section{Preliminaries}

We collect some basic notation and results on Schwarz symmetrals, star-shaped sets, the fractional Laplacian operator, $L^2$ polar projection bodies and  radial mean bodies. The books of Gardner \cite{Gar}, Schneider \cite{Sch} and Lieb and Loss \cite{LL} are good general references.

\subsection{Symmetrization} For $E\subset\rn$, the characteristic function $\chi_E$ is defined by $\chi_{E}(x)=1$ for $x\in E$, and $\chi_E(x)=0$ for $x\notin E$. If $E$ is a Borel set of finite measure, the Schwarz symmetral of $E$, denoted by $E^{\star}$, is the centered Euclidean ball with the same volume as $E$. We denote the Euclidean unit ball in $\rn$ as $B^n$. 

Let $f:\rn\rightarrow \tr$ be a non-negative measurable function. The superlevel set of $f$ is denoted by $\{f\ge t\}:=\{x\in\rn: f(x)\ge t\}$ for $t> 0$. We say $f$ is non-zero if $\{f\neq 0\}$ has positive measure. For $f$ with superlevel sets of finite measure, by the layer cake formula,
  \[f(x)=\int_0^{\infty}\chi_{\{f\ge t\}}(x)dt\]
for almost every $x\in\rn$. 

The Schwarz symmetral of $f$, denoted by $f^{\star}$, is defined as
\[f^{\star}(x)=\int_0^{\infty}\chi_{\{f\ge t\}^{\star}}(x)dt\]
for $x\in\rn$. Hence $f^{\star}$ is radially symmetric and the superlevel set $\{f^{\star}\ge t\}$ has the same volume as $\{f\ge t\}$ for every $t> 0$. The Schwarz symmetral $f^{\star}$ is also called the symmetric decreasing rearrangement of $f$. If $|x|<|y|$ implies $f^{\star}(x)>f^{\star}(y)$, we say $f^{\star}$ is strictly symmetric decreasing.

One crucial technique in our proofs is the Riesz rearrangement inequality (see, for example  \cite[Theorem 3.7]{LL}) and the characterization of equality case due to Burchard  \cite{Bu}.

\begin{thm}[Riesz's rearrangement inequality]\label{RRI}
    For $f,g,k:\rn\rightarrow\tr$ non-negative, measurable functions with superlevel sets of finite measure,
    \[\di f(x)k(x-y)g(y)dxdy\leq \di f^{\star}(x)k^{\star}(x-y)g^{\star}(y)dxdy.\]
\end{thm}

\begin{thm}[Burchard]\label{RRI B}
    Let $f,g,k:\rn\rightarrow \tr$ non-negative, non-zero, measurable functions with superlevel sets of finite measure such that
    \[\di f(x)k(x-y)g(y)dxdy<\infty.\]
    If at least two of the Schwarz symmetrals $f^{\star}, g^{\star}, k^{\star}$ are strictly symmetric decreasing, then there is equality in
    \[\di f(x)k(x-y)g(y)dxdy\leq \di f^{\star}(x)k^{\star}(x-y)g^{\star}(y)dxdy,\]
    if and only if there is a volume-preserving $\phi\in{\rm GL}(n)$ and $a,b,c\in\rn$ with $c=a+b$ such that
    \[f(x)=f^{\star}(\phi^{-1}x-a), k(x)=k^{\star}(\phi^{-1}x-b), g(x)=g^{\star}(\phi^{-1}x-c),\]
    for $x\in\rn$.
\end{thm}

\begin{thm}[Burchard]\label{RRI B2}
    Let $A,B$ and $C$ be sets of finite positive measure in $\rn$ and denote by $\alpha$, $\beta$ and $\gamma$ the radii of their Schwarz symmetrals $A^{\star}$, $B^{\star}$ and $C^{\star}$. For $|\alpha-\beta|<\gamma<\alpha+\beta$, there is equality in
    \[\di\chi_A(x)\chi_B(x-y)\chi_C(y)dxdy\leq \di \chi_{A^{\star}}(x)\chi_{B^{\star}}(x-y)\chi_{C^{\star}}(y)dxdy,\]
    if and only if, up to sets of measure zero,
    \[A=a+\alpha D,~~B=b+\beta D,~~C=c+\gamma D,\]
    where $D$ is a centered ellipsoid, and $a$, $b$ and $c=a+b$ are vectors in $\rn$.
\end{thm}

%Denote the class of Borel sets in $\rn$ as $\mathcal{B}(\rn)$. For $f\in L^1(\rn)$, it is said of bounded variation on $\rn$ if there exists a finite vector-valued Radon measure ${\rm D}f:\mathcal{B}(\rn)\rightarrow \rn$ such that
%\[\int_{\rn} \left\langle\phi(x), {\rm D}f(x)\right\rangle dx=-\int_{\rn} f(x){\rm div} \phi(x)dx,\]
%for every $\phi\in C_c^{\infty}(\rn; \rn)$, where $C_c^{\infty}(\rn; \rn)$ is the set of smooth vector fields $\phi:\rn\rightarrow \rn$ with compact support and ${\rm div} \phi$ is the divergence of $\phi$. We write $BV(\rn)$ as the set of $L^1$ functions with bounded variation. 

%BV functions are closely related to sets of finite perimeter. For $E\subset \rn$ with finite perimeter, the characteristic function $1_E$ is of bounded variation. For $f\in BV(\rn)$, the superlevel set $\{f>t\}$ is of finite perimeter, for almost all $t$. 

\subsection{Star-shaped sets and dual mixed volumes.} 
We say $L\subset \rn$ is a star-shaped set (with respect to the origin) if for any $x\in L$ the line segment $[0,x]$ is contained in $L$. The gauge function $\|\cdot\|_L:\rn\rightarrow [0,+\infty]$ of a star-shaped set $L$ is defined by
\[\|x\|_L=\inf\{\lambda >0: x\in \lambda L\},\]
and the radial function $\rho_L:\rn\backslash\{0\}\rightarrow [0,+\infty]$ is
\[\rho_{L}(x)=\|x\|_L^{-1}=\sup\{\lambda\ge 0: \lambda x\in L\}.\]
In particular, $\|\cdot\|_{B^n}$ is the Euclidean norm $|\cdot|$. A star-shaped set $L$ is said to be a star body if $\rho_{L}$ is continuous and strictly positive on $\rn\backslash\{0\}$. 
%We say that two star-shaped sets are dilates if  $\rho_K(u)=c\rho_L(u)$ for some $c>0$ and almost all $u\in\sn$.

For a star-shaped set $L$ with measurable radial function, the $n$-dimensional volume, or the $n$-dimensional Lebesgue measure of $L$ is given by
\[|L|=\frac{1}{n}\int_{\sn}\rho_{L}(u)^ndu,\]
where $du$ is the spherical Lebesgue measure. The dual mixed volume for star bodies $K$ and $L$ by Lutwak  \cite{Lut} is given by
\[\tilde{V}_{\alpha}(K,L)=\frac{1}{n}\int_{\sn}\rho_{K}(u)^{n-\alpha}\rho_L(u)^{\alpha}du,\]
for $\alpha\in\tr\backslash\{0,n\}$. This notion extends naturally to star-shaped sets by restricting the integration to the subset of $\sn$, where both $\rho_K$ and $\rho_L$ are positive. For star bodies $K$ and $L$,
\[\tilde{V}_{\log}(K,L)=\frac{1}{n|K|}\int_{\sn}\rho_K(u)^n\log\Big(\frac{\rho_L(u)}{\rho_K(u)}\Big)du\]
where %Note that $\tilde{V}_{\alpha}(K,K)=|K|$. %The definition is first  given by Lutwak in  \cite{Lut}. 
\[\tilde{V}_{\log}(K,L)=\lim_{\alpha\rightarrow 0}\log\Big(\frac{\tilde{V}_{\alpha}(K,L)}{|K|}\Big)^{1/\alpha}.\]
Similarly, $\tilde{V}_{\log}(K,L)$ can be extended to star-shaped sets $K$ and $L$. This notion was further extended by Hou and Xiao \cite{HX} via the anisotropic logarithmic potential.

For $\alpha\in(0,n)$, and star-shaped sets $K$ and $L$ of positive volume, the dual mixed volume inequality states that
\[\tilde{V}_{\alpha}(K,L)\leq |K|^{\frac{n-\alpha}{n}}|L|^{\frac{\alpha}{n}},\]
and the inequality reversed for $\alpha<0$ and $\alpha>n$. In particular, for $\alpha\rightarrow 0$,
%, which is given by
%\[\tilde{V}_{\alpha}(K,L)\ge |K|^{\frac{n-\alpha}{n}}|L|^{\frac{\alpha}{n}}.\]
\begin{equation}\label{dual mix vol ineq-log}
    \tilde{V}_{\log}(K,L)\leq \frac{1}{n}\log\Big(\frac{|L|}{|K|}\Big),
\end{equation}
which is a direct corollary from  \cite[Theorem 6.1]{GHWY}. For all the cases, the equality holds if and only if $K$ and $L$ are dilates, that is, $\rho_K(u)=c\rho_L(u)$ for some $c>0$ and almost all $u\in\sn$.

\subsection{Fractional \texorpdfstring{$\boldsymbol{L^2}$}{} polar projection bodies} For $-1<\alpha<0$ and measurable $f:\rn\rightarrow\tr$, the $(-\alpha)$-fractional $L^2$ polar projection body of $f$, denoted by $\Pi_2^{*,-\alpha}f$, is defined in \cite{HL2} by its radial function as
\[\rho_{\Pi_2^{*,-\alpha}f}(u)^{2\alpha}(u)=\int_0^{\infty}t^{2\alpha-1}\int_{\rn}|f(x+tu)-f(x)|^2dxdt,\]
where $u\in\sn$.

The set $\Pi_2^{*,-\alpha}f$ is a star body with the origin in its interior and there exists $c>0$, which only depends on $f$, such that $\Pi_2^{*,-\alpha}f\subset cB^n$ (See \cite[Proposition 4]{HL2}). For volume-preserving linear transformations $\phi:\rn\rightarrow\rn$,
\[\Pi_2^{*,-\alpha}(f\circ \phi^{-1})=\phi\Pi_2^{*,-\alpha}f.\]

The affine fractional $L^2$ Sobolev inequalities \cite[Theorem 1]{HL2} states that
\begin{equation}\label{aff frac l2}
    2\gamma_{n,2\alpha}\|f\|_{\frac{2n}{n+2\alpha}}^2\leq n\omega_n^{\frac{n-2\alpha}{n}}|\Pi_2^{*,-\alpha}f|^{\frac{2\alpha}{n}}\leq \di\frac{|f(x)-f(y)|^2}{|x-y|^{n-2\alpha}}dxdy
\end{equation}
%MR2197182 - A sharp Sobolev trace inequality for the fractional-order derivatives Xiao, Jie, Bull. Sci. Math. 130 (2006), no. 1, 87-96. (Reviewer:Horváth, J.)
for $f\in W^{-\alpha,2}(\rn)$ and $-1<\alpha<0$. Removing the intermediate term gives the classical fractional $L^2$ Sobolev inequalities; moreover, the inequality between the first and the intermediate terms in \eqref{aff frac l2} may be viewed as a variant of the Sobolev trace inequality in \cite[Theorem 1.1]{Xiao06}. The constant $\gamma_{n,2\alpha}$ is given by \eqref{const 1}, where $\Gamma(\alpha)$ is given by the analytic continuation formula 
    \begin{equation}\label{gn}
        \Gamma(\alpha)=\left\{
        \begin{aligned}
            &\int_0^{\infty}t^{\alpha-1}e^{-t}dt&&\text{for}~~\alpha>0,\\
            &\int_0^{\infty}t^{\alpha-1}(e^{-t}-1)dt&&\text{for}~~-1<\alpha<0.
        \end{aligned}\right.
    \end{equation}
The optimal constant $2\gamma_{n,2\alpha}$ is characterized by Cotsiolis and Tavoularis \cite{CT}.
%and
%\[W^{-\alpha,2}(\rn)=\Big\{f\in L^2(\rn):\di\frac{|f(x)-f(y)|^2}{|x-y|^{n-2\alpha}}dxdy<\infty\Big\}.\]

\subsection{Fractional Laplacian} Leoni's book  \cite{Leoni} is a good reference for fractional Sobolev spaces and fractional Laplacian. For $\alpha>-1$, the fractional Laplacian $(-\Delta)^{-\alpha}$ is defined by
\[\widehat{(-\Delta)^{-\alpha} f}(\xi)=(2\pi|\xi|)^{-2\alpha}\widehat{f}(\xi).\]

For $\alpha\in(-1,0)$, 
\begin{equation}\label{pv}
    (-\Delta)^{-\alpha}f(x)=\frac{2^{-2\alpha}\Gamma(\frac{n}{2}-\alpha)}{\pi^{\frac{n}{2}}|\Gamma(\alpha)|}\lim_{r\rightarrow 0^+}\int_{\rn\backslash B(x,r)}\frac{f(x)-f(y)}{|x-y|^{n-2\alpha}}dy,
\end{equation}
where $B(x,r)$ is the ball centered at $x$ with radius $r$. For $\alpha\in(0,n/2)$, the fractional Laplacian is the Riesz potential given by
\begin{equation}\label{Riesz potential}
(-\Delta)^{-\alpha}f(x)=\frac{2^{-2\alpha}\Gamma(\frac{n}{2}-\alpha)}{\pi^{\frac{n}{2}}|\Gamma(\alpha)|}\int_{\rn}\frac{f(y)}{|x-y|^{n-2\alpha}}dy.    
\end{equation}

\subsection{Radial mean bodies}
Let $K\subset\rn$ be a convex body, that is, compact convex sets with non-empty interior. For $\alpha>-1$, Gardner and Zhang \cite{GZ} defined the radial $\alpha$-th mean body %$\rR_{\alpha} K$ is defined 
by
\[\rho_{\rR_{\alpha} K}(u)=\Big(\frac{1}{|K|}\int_K\rho_{K-x}(u)^{\alpha}dx\Big)^{1/\alpha},\]
for $\alpha\neq 0$, and
\[\log\rho_{\rR_0 K}(u)=\frac{1}{|K|}\int_K\log\rho_{K-x}(u)dx,\]
for each $u\in\sn$.

Haddad and Ludwig \cite{HL2} found that $\rR_{\alpha}K$ is proportional to ${\rm S}_{\alpha}\chi_K$ for $\alpha>0$ and proportional to $\Pi_2^{*,\alpha/2}\chi_K$ for $\alpha\in(-1,0)$. Then by a certain normalization, the $\alpha$-th radial mean body $\Ra f$ for non-negative $f\in L^2(\rn)$ is defined in the following way.

For $\alpha>0$,
\[\rho_{\Ra f}(u)=\Big(\frac{\alpha}{\|f\|_2^2}\Big)^{1/\alpha}\rho_{\Sa}(u),\]
and for $-1<\alpha<0$,
\[\rho_{\Ra f}(u)=\Big(\frac{-2\alpha}{\|f\|_2^2}\Big)^{1/\alpha}\rho_{\Pi_2^{*,-\alpha/2}f}(u).\]
For $\alpha=0$, by taking limits of $\rho_{\Ra f}(u)$ as $\alpha\rightarrow 0$,
\[\log\rho_{\rR_0 f}(u)=-\gamma+\int_0^{\infty}\frac{1}{t}\Big(\frac{1}{\|f\|_2^2}\int_{\rn}f(x)f(x+tu)dx-e^{-t}\Big)dt,\]
which will be introduced in detail in Section 4. %and we will explain why the definition is given in this way.

An important property of $\Ra f$ is the volume $|\Ra f|$ remains non-decreasing after the Schwarz symmetrization, that is $|\Ra f|\leq |\Ra f^{\star}|$, for both $\alpha>0$ and $-1<\alpha<0$ (See \cite[Lemma 10]{HL3} and \cite[Theorem 14]{HL2}).

\section{The affine logarithmic HLS inequality}

%The logarithmic HLS inequality is a limiting case of Lieb's sharp HLS inequalities ($\alpha$ tending to $n$). This affine case could also be considered as the limiting case of affine HLS inequalities as $\alpha\rightarrow n$. We apply the rearrangement inequality to the proof.
In this section, we always consider non-negative $f\in L^1(\rn)$ such that 
\begin{equation}\label{cond1}
    \|f\|_1=1, \int_{\rn}f(x)\log (1+|x|^2)dx<\infty~~ \text{and}~~\int_{\rn}f(x)\log f(x)dx~~\text{is finite}.
\end{equation}
Recall that
\[\rho_{\Sa}(u)^{\alpha}=\int_0^{\infty}t^{\alpha-1}\int_{\rn}f(x)f(x+tu)dxdt,\]
and thus we have
\begin{equation*}
\begin{aligned}
-\di f(x)&\log\|x-y\|_{K}f(y)dxdy\\
&=-\di f(x)\log |x-y|f(y)dxdy+\int_{\sn}\rho_{\Snf}(u)^n\log \rho_{K}(u)du.
\end{aligned}
\end{equation*}
%and $|\Snf|=\frac{1}{n}$, then by the dual Orlicz Minkowski inequality,
%\[-\int_{\sn}\rho_{\Snf}(u)^n\log \rho_{\Snf}(u)du=\frac{1}{n}\int_{\sn}\rho_{\Snf}(u)^n\log\left(\frac{1}{\rho_{\Snf}(u)^n}\right)du\leq\frac{1}{n}\log(n\omega_n) ,\]
%which implies the first inequality in Theorem 1. 
By the dual mixed volume inequality \eqref{dual mix vol ineq-log} and the fact that $n|\Snf|=\|f\|_1^2=1$,
\[\int_{\sn}\rho_{\Snf}(u)^n\log\rho_K(u)du\leq \int_{\sn}\rho_{\Snf}(u)^n\log\rho_{\Snf}(u)du+\frac{1}{n}\log(n|K|)\]
which directly implies the following result.

\begin{corollary}\label{tech}
Suppose $f\in L^1(\rn)$ is non-negative and satisfies the assumption \eqref{cond1}. Then
   % \begin{equation}\label{ex rri}
   % -\int_{\rn}\int_{\rn}f(x)\log \|x-y\|_{\Snf}f(y)dxdy\leq -\di f^{\star}(x)\log \|x-y\|_{{\rm S}_n f^{\star}}f^{\star}(y)dxdy,
   % \end{equation}
   % and
    \begin{equation*}\label{sup}
    \begin{split}
            \sup\left\{-\di f(x)\log\|x-y\|_Kf(y)dxdy: K\subset\rn~~\text{star-shaped},~~|K|=\omega_n \right\}\\
    =-\di f(x)\log\|x-y\|_{\Snf}f(y)dxdy+\frac{1}{n}\log(n\omega_n),
    \end{split}
    \end{equation*}
    where the supreme is attained if and only if $K$ and $\Snf$ are dilates.
\end{corollary}

This corollary directly implies the first inequality in Theorem 1. For the second inequality, we ask for the rearrangement inequality for the affine logarithmic kernel, which is an extension of the following lemma established by Carlen and Loss.

\begin{lem} \cite[Lemma 2]{CL}\label{CL}
Suppose $f\in L^1(\rn)$ is non-negative and satisfies the assumption \eqref{cond1}. 
%nonnegative integrable function such that $\int_{\rn}f(x)\log f(x)dx$ and $\int_{\rn} f(x)\log(1+|x|^2)dx$ are finite. 
Then the following holds:
    \[\int_{\rn}f(x)\log f(x)dx=\int f^{\star}(x)\log f^{\star}(x)dx\]
    \begin{equation}\label{rri log}
    -\di f(x)\log |x-y|f(y)dxdy\leq -\di f^{\star}(x)\log |x-y|f^{\star}(y)dxdy.
    \end{equation}
In particular the right-hand side of \eqref{rri log} is well-defined. Furthermore the equality sign in \eqref{rri log} holds if and only if $f$ is a translate of $f^{\star}$.
\end{lem}

The main idea of the proof is decomposing $-\log r$ into the sum of two decreasing functions. The function 

\begin{equation*}
    k(r)=\left\{
    \begin{aligned}
    &-\log r &\text {if}~~r\leq 1\\
    &-\log 2-2\log r+\log(1+r^2)&\text {if}~~r>1
    %&-\int_{1}^rs^{-1}\frac{2}{1+s^2}ds &{\text if}~~r>1
    \end{aligned}
    \right.
\end{equation*}
is strictly decreasing and $\lim_{r\rightarrow\infty}k(r)=-\log2$. Define $h_1(r)=k(r)+\log2$, which is positive and strictly decreasing. We then have the decomposition
\begin{equation}\label{dcmp}
    -\log r=h_1(r)+h_2(r).
\end{equation}
Note that $h_2(r)=-\log2$ for $r\leq 1$ and $h_2(r)=-\log(r+1/r)$ for $r>1$, which implies that $h_2$ is also a decreasing function.

By employing this decomposition, we establish the following anisotropic version of Lemma \ref{CL}, where the Euclidean norm is replaced by the gauge function of a star-shaped set $K$. This approach follows the anisotropic fractional norms introduced by Ludwig \cite{Lud1,Lud2}, which were further developed in \cite{HL1,HL2,HL3}.

%By employing this decomposition, we establish the following anisotropic version, where we replace the Euclidean norm with the gauge function of a star-shaped set. The anisotropic fractional norms were introduced by Ludwig\cite{Lud1, Lud2} and applied in \cite{HL1, HL2, HL3} as well.

\begin{lem}\label{ex rri2}
    Suppose $f\in L^1(\rn)$ is non-negative and satisfies \eqref{cond1}. Assume $K\subset \rn$ is a {star-shaped set} with measurable radial function and  $|K|>0$. Then
    \begin{equation}\label{ani rri}
        -\di f(x)\log\|x-y\|_K f(y)dxdy\leq -\di f^{\star}(x)\log\|x-y\|_{K^{\star}}f^{\star}(y)dxdy.
    \end{equation}
    For $f^{\star}$ strictly symmetric decreasing, there is equality if and only if 
    \begin{equation}\label{eq cond}
    f(x)=f^{\star}(\phi^{-1}x-x_0),~~\text{and}~~K=\phi B,
    \end{equation}
    where $\phi\in {\rm SL}(n)$, $x_0\in\rn$ and $B$ is a centered ball.
\end{lem}

\begin{proof}
    We use \eqref{dcmp} and define $H_i(x)=h_i(\|x\|_K)$, for $i=1,2$. Thus
    \[-\log\|x-y\|_K=H_1(x-y)+H_2(x-y).\]

    Since replacing $K$ with $cK$ for some $c>0$, does not affect the validity of the inequality \eqref{ani rri}, we may assume $|K|=\omega_n$ without loss of generality. Hence, 
    \[H_1^{\star}(x)=h_1(|x|).\]
    Recall that $h_1$ is positive and strictly decreasing. Hence the superlevel set $\{x:H_1(x)\ge t\}=h_1^{-1}(t)K$ has finite measure for every $t>0$. Similarly, the superlevel set of $H_2$ also has finite measure.

    By the Riesz's rearrangement inequality, we have
    \begin{equation}\label{h1}
    \di f(x)H_1(x-y)f(y)dxdy\leq \di f^{\star}(x)h_1(|x-y|)f^{\star}(y)dxdy.
    \end{equation}
    Before applying Theorem \ref{RRI B}, it is necessary to show that the right-hand side is finite. By the definition of $h_1$, it suffices to show
    \begin{equation}\label{finite}
    \di f^{\star}(x)k(|x-y|)f^{\star}(y)dxdy<\infty.
    \end{equation}
Note that, $k(r)=-\log r+\tilde{k}(r)$, where
\begin{equation*}
    \tilde{k}(r)=\left\{
    \begin{aligned}
    &0 &\text {if}~~r\leq 1\\
    &-\log 2-\log r+\log(1+{r^2})&\text {if}~~r>1
    %&-\int_{1}^rs^{-1}\frac{2}{1+s^2}ds &{\text if}~~r>1
    \end{aligned}\right.
\end{equation*}
and for $r>1$, 
\[\tilde{k}(r)\leq\log(1+r^2)-\log2\leq 2\log r.\]

Hence,
\begin{equation}\label{step}
    \begin{aligned}
    \di f^{\star}(x)&\tilde{k}(|x-y|)f^{\star}(y)dxdy\\
    &\leq2\iint_{|x-y|>1}f^{\star}(x)\log|x-y|f^{\star}(y)dxdy\\
    &\leq \iint_{|x-y|>1}f^{\star}(x)\left(\log 2+\log(1+|x|^2)+\log(1+|y|^2)\right)f^{\star}(y)dxdy<\infty.
\end{aligned}
\end{equation}

Since $k(r)=-\log r+\tilde{k}(r)$, we have
\begin{equation}\label{decomp}
    \begin{aligned}
        \di f^{\star}&(x)k(|x-y|)f^{\star}(y)dxdy\\
        &=-\di f^{\star}(x)\log|x-y|f^{\star}(y)dxdy+\di f^{\star}(x)\tilde{k}(|x-y|)f^{\star}(y)dxdy.
    \end{aligned}
\end{equation}
By the logarithmic HLS inequality \eqref{log HLS}, together with \eqref{step}, we obtain \eqref{finite}.

    Since $f^{\star}$ is strictly symmetric decreasing as assumed, by Theorem \ref{RRI B} the equality holds in \eqref{h1} precisely when
    \begin{equation}\label{eq cond'}
        f(x)=f^{\star}(\phi^{-1}x-x_0), H_1(x)=h_1(|\phi^{-1}x|),
    \end{equation}
    for some $\phi \in{\rm SL}(n)$ and $x_0\in\rn$.

    Note that $h_2$ is not positive and not bounded from below. Define for $R>0$,
    \begin{equation*}
    q_R(r)=\left\{
        \begin{aligned}
            &h_2(r)-h_2(R) &\text {if}&~~r\leq R\\
            &0 &\text {if}&~r>R,
        \end{aligned}
        \right.
    \end{equation*}
    and $Q_R(x)=q_R(\|x\|_K)$. 
    By the Riesz's rearrangement inequality, there is
    \[\di f(x)Q_R(x-y)f(y)dxdy\leq \di f^{\star}(x)q_R(|x-y|)f(y)dxdy.\]
    Recall that $\|f\|_1=1$, and by the monotone convergence theorem, we have
    \[\di f(x)H_2(x-y)f(y)dxdy=\lim_{R\rightarrow\infty}\Big(\di f(x)Q_R(x-y)f(y)dxdy+h_2(R)\Big).\]    
    Together with $\|f\|_1=\| f^{\star}\|_1=1$, we obtain
    \[\di f(x)H_2(x-y)f(y)dxdy\leq \di f^{\star}(x)h_2(|x-y|)f^{\star}(y)dxdy.\]
    Combined with inequality \eqref{h1}, this inequality yields the desired result. 

    If equality holds in \eqref{ani rri}, then \eqref{h1} holds with equality as well. Therefore by \eqref{eq cond'},
     \[f(x)=f^{\star}(\phi^{-1}x-x_0),\]
    for some $\phi \in {\rm SL}(n)$ and $x_0\in\rn$, and hence $K=\phi B^n$.
\end{proof}

\begin{proof}[Proof of Theorem 1]
    By Corollary \ref{tech}, the first inequality holds, and equality is attained if and only if $\Snf$ is a ball, which happens when $f$ is  radially symmetric.

    For the second inequality, by Lemma \ref{ex rri2} and applying logarithmic HLS inequality to $f^{\star}$, we have
     \begin{equation*}
        \begin{aligned}
            -\di f(x)\log\|x-y\|_{\Snf}&f(y)dxdy+\frac{1}{n}\log(n\omega_n)\\
            \leq&-\di f^{\star}(x)\log\|x-y\|_{\rS_n f^{\star}}f^{\star}(y)dxdy+\frac{1}{n}\log(n\omega_n)\\
            = &-\di f^{\star}(x)\log |x-y| f^{\star}(y)dxdy\\
            \leq&\frac{1}{n}\int_{\rn}f^{\star}(x)\log f^{\star}(x)dx+\gamma_n=\frac{1}{n}\int_{\rn}f(x)\log f(x)dx+\gamma_n,
        \end{aligned}
    \end{equation*}
    which yields the second inequality.

    If equality holds, then
    \[-\di f^{\star}(x)\log |x-y|f^{\star}(y)dxdy=\frac{1}{n}\int_{\rn}f^{\star}(x)\log f^{\star}(x)dx+\gamma_n.\]
    Therefore $f^{\star}$ is the optimizer of the sharp logarithmic HLS inequality, that is, $f^{\star}(x)=a(1+\lambda|x|^2)^{-n}$, where $a,\lambda>0$ are such that $\|f\|_1=1$. 
    
    Since $f^{\star}(x)=a(1+\lambda|x|^2)^{-n}$ is strictly symmetric decreasing, we apply Lemma \ref{ex rri2} to obtain the equality case in the second inequality.
\end{proof}

\section{A Beckner-type affine logarithmic Sobolev inequality}
%\begin{lem} \cite[Lemma 3]{Lud1}\label{Lud1}
%    If $A\subset \tr$ is a bounded Borel set of finite perimeter, then
%    \[\lim_{\alpha\rightarrow 0^+}\alpha P_{\alpha}(A)  %\int_{A}\int_A|x-y|^{\alpha-1}dxdy
%    =2|A|,\]
%    and
%    \[P_{\alpha}\leq\frac{4}{\alpha}\max\{1,{\rm diam}(A)\}+{\rm diam}(A)^2+P_{\alpha^{'}}(A)\]
%    for $0<\alpha$
%\end{lem}

Recall that the affine HLS inequalities are given by,
\[\gamma_{n,\alpha}\|f\|_{\frac{2n}{n+\alpha}}^2\ge n\omega_n^{\frac{n-\alpha}{n}}|\Sa|^{\frac{\alpha}{n}}\ge\di\frac{f(x)f(y)}{|x-y|^{n-\alpha}}dxdy,~~\alpha\in(0,n).\]
By duality, Carlen \cite{Ca} shows that the HLS inequalities are equivalent to the fractional $L^2$ Sobolev inequalities, which are given by
\[2\gamma_{n,2\alpha}\|f\|_{\frac{2n}{n+2\alpha}}^2\leq\int_{\rn}\int_{\rn}\frac{|f(x)-f(y)|^2}{|x-y|^{n-2\alpha}}dxdy,~~\alpha\in(-1, 0).\]

%\[\gamma_{n,\alpha}=\pi^{\frac{n-\alpha}{2}}\frac{\Gamma(\frac{\alpha}{2})}{\Gamma (\frac{n+\alpha}{2})}\left(\frac{\Gamma(n)}{\Gamma(\frac{n}{2})}\right)^{\frac{\alpha}{n}}.\]

Ma${\rm z}^{\prime}$ya and Shaposhnikova  \cite{MS} give the limiting behavior for $\alpha\rightarrow 0^-$ of the fractional $L^2$ Sobolev inequalities, which says for $f\in \bigcup_{-1<\alpha<0}W^{-\alpha,2}(\rn)$,
\begin{equation*}\label{lim l2 0}
    \lim_{\alpha\rightarrow 0^-}-\alpha\int_{\rn}\int_{\rn}\frac{|f(x)-f(y)|^2}{|x-y|^{n-2\alpha}}dxdy=\lim_{\alpha\rightarrow 0^-}2\gamma_{n,2\alpha}\|f\|_{\frac{2n}{n+2\alpha}}^2= n\omega_n\|f\|_2^2.
\end{equation*}
Hence the affine fractional $L^2$ Sobolev inequalities \eqref{aff frac l2} become an identity as $\alpha\rightarrow 0^-$.

The following lemma shows that the affine HLS inequalities also become an identity as $\alpha\rightarrow 0^+$ for non-negative $f\in L^1(\rn)\cap L^2(\rn)$.

{\begin{lem}\label{lim}
Suppose $f\in L^1(\rn)\cap L^2(\rn)$ is non-negative. Then
%Let $f\in \bigcup_{0<\beta<\frac{1}{2}}W^{\beta,2}(\rn)$ be non-negative, then 
    \begin{equation}\label{lim hls 0 lemma}
            \lim_{\alpha\rightarrow 0^+}\alpha\di\frac{f(x)f(y)}{|x-y|^{n-\alpha}}dxdy=\lim_{\alpha\rightarrow 0^+}\alpha\gamma_{n,\alpha}\|f\|_{\frac{2n}{n+\alpha}}^2=n\omega_n\|f\|_2^2,
    \end{equation}
and therefore
\[\lim_{\alpha\rightarrow 0^+}\alpha n\omega_n^{\frac{n-\alpha}{n}}|\Sa|^{\frac{\alpha}{n}}=n\omega_n\|f\|_2^2.\]
\end{lem}}

\begin{proof} \noindent{Step 1:} We first consider the limit of $\alpha\gamma_{n,\alpha}\|f\|_{\frac{2n}{n+\alpha}}^2$. 

Note that $\lim_{\alpha\rightarrow 0^+}\alpha\gamma_{n,\alpha}=n\omega_n$. Hence, it suffices to consider the limit of $\|f\|_{\frac{2n}{n+\alpha}}^2$. Let
    \[A=\{x\in\rn:f(x)\ge 1\}~~\text{and}~~B=\{x\in\rn: f(x)<1\},\]
    and $(\alpha_i)$ be a non-negative sequence converging to $0$ as $i\rightarrow\infty$. For $x\in A$,
    \[f(x)^{\frac{2n}{n+\alpha_i}}\leq f(x)^2,\]
    together with dominated convergence theorem, we have
    \begin{equation}\label{lim A}
            \lim_{i\rightarrow \infty}\int_A f(x)^{\frac{2n}{n+\alpha_i}}dx=\int_A f(x)^2dx.
    \end{equation}

    %Since $f\in W^{\beta,2}(\rn)$ for some $\beta\in(0,\frac{1}{2})$, by the fractional $L^2$ Sobolev inequality, there is
    %\[\|f\|_{\frac{2n}{n-2\beta}}^2\leq \sigma_{n,-\beta}\di \frac{|f(x)-f(y)|^2}{|x-y|^{n+2\beta}}dxdy<\infty,\]
    %and therefore $f\in L^{\frac{2n}{n-2\beta}}(\rn)$. 
    For $x\in B$,
    \[f(x)^{\frac{2n}{n+\alpha_i}}\leq f(x).\]
    By the dominated convergence theorem, we have
    \[\lim_{i\rightarrow\infty}\int_B f(x)^{\frac{2n}{n+\alpha_i}}dx=\int_B f(x)^2dx.\]
    Combined with \eqref{lim A}, we obtain $\lim_{\alpha\rightarrow 0^+}\alpha\gamma_{n,\alpha}\|f\|_{\frac{2n}{n+\alpha}}^2=n\omega_n\|f\|_2^2$.
    
    \noindent{Step 2:} We consider the first limit in \eqref{lim hls 0 lemma}.

    By step 1, $f\in L^{\frac{2n}{n+\alpha}}(\rn)$ for sufficiently small $\alpha$. Using the fractional Laplacian operator introduced in Section 2, \eqref{Riesz potential} and the Parseval's formula \cite[Theorem 5.3]{LL}, we can rewrite the double integral as
    \begin{equation}\label{riesz potential}
         \alpha\di\frac{f(x)f(y)}{|x-y|^{n-\alpha}}dxdy=2\pi^{\frac{n}{2}-\alpha}\frac{\Gamma(\frac{\alpha}{2}+1)}{\Gamma(\frac{n-\alpha}{2})}\int_{\rn}|\widehat{f}(\xi)|^2\frac{d\xi}{|\xi|^{\alpha}},
    \end{equation}
    {which can also be found in \cite{Xiao06}}. By Fatou's lemma and the fact $\|f\|_2=\|\hat f\|_2$, 
       \[n\omega_n\|f\|_2^2\leq \liminf_{\alpha\rightarrow 0^+}2\pi^{\frac{n}{2}-\alpha}\frac{\Gamma(\frac{\alpha}{2}+1)}{\Gamma(\frac{n-\alpha}{2})}\int_{\rn}|\widehat{f}(\xi)|^2\frac{d\xi}{|\xi|^{\alpha}}.\]
    Combined with \eqref{riesz potential}, step 1 and the affine HLS inequalities, we obtain then
    \[\limsup_{\alpha\rightarrow 0^+}2\pi^{\frac{n}{2}-\alpha}\frac{\Gamma(\frac{\alpha}{2}+1)}{\Gamma(\frac{n-\alpha}{2})}\int_{\rn}|\widehat{f}(\xi)|^2\frac{d\xi}{|\xi|^{\alpha}}\leq \lim_{\alpha\rightarrow 0^+}\alpha\gamma_{n,\alpha}\|f\|_{\frac{2n}{n+\alpha}}^2=n\omega_n\|f\|_2^2, \]
    which concludes the proof.
    \end{proof}

    As introduced in Section 2.5, using a normalization of $\Sa$ and $\Pit$, Haddad and Ludwig \cite{HL3} defined the radial mean body for non-zero and non-negative function $f\in L^2(\rn)$. 
    The following result by Haddad and Ludwig implies the continuity of $\rho_{\Ra f}(u)$ with respect to $\alpha$ and gives the definition of $\rho_{\rR_0 f}(u)$.

    \begin{lem} \cite[Lemma 17]{HL3}\label{HL}
        Let $\omega:[0,\infty)\rightarrow [0,\infty)$ be decreasing with
        \[0<\int_0^{\infty}t^{\alpha-1}\omega(t)dx<\infty\]
        for $\alpha>0$ and
        \[0<\int_0^{\infty}t^{\alpha-1}(\omega(0)-\omega(t))dt<\infty\]
        for $-1<\alpha<0$. If $\phi:[0,\infty)\rightarrow [0,\infty)$ is non-zero, with $\phi(0)=0$, and such that $t\mapsto\phi(t)$ and $t\mapsto\phi(t)/t$ are increasing on $(0,\infty)$, then
        \begin{equation*}
            \zeta(\alpha)=\left\{
                \begin{aligned}
                &\left(\frac{\int_0^{\infty}t^{\alpha-1}\omega(\phi(t))dt}{\int_0^{\infty}t^{\alpha-1}\omega(t)dt}\right)^{1/\alpha}&&\text{for}~~\alpha>0,\\
                &\exp\left(\int_0^{\infty}\frac{\omega(\phi(t))-\omega(t)}{t\omega(0)}dt\right)&&\text{for}~~\alpha=0,\\
                &\left(\frac{\int_0^{\infty}t^{\alpha-1}(\omega(\phi(t))-\omega(0))dt}{\int_0^{\alpha-1}(\omega(t)-\omega(0))dt}\right)^{1/\alpha}&&\text{for}~~-1<\alpha<0.
                \end{aligned}\right.
        \end{equation*}
        is a continuous, decreasing function of $\alpha$ on $(-1,\infty)$. Moreover, $\zeta$ is constant on $(-1,\infty)$ if $\phi(t)=\lambda t$ on $[0,\infty)$ for some $\lambda>0$.
    \end{lem}

    For a non-zero and log-concave function $f\in L^2(\rn)$ and $u\in\sn$, set
    \begin{equation}\label{g-def}
        g(t,u)=\frac{1}{\|f\|_2^2}\int_{\rn}f(x)f(x+tu)dx.
    \end{equation}
   For simplicity of notation, we shall often write $g(t)$ instead of $g(t,u)$ when $u\in\sn$ is fixed and no confusion arises. This convention will be used throughout this section.
   
   Let $u\in\sn$ be fixed. Then $g(t)$ is a positive, decreasing and log-concave function by the Pr{\'e}kopa-Leindler inequality. Let $\omega(t)=g(0)e^{-t}$ and $\phi(t)=-\log(g(t)/g(0))$.  Here we apply the analytic continuation formulas \eqref{gn} for the Gamma function. 
    We observe that
    \[\zeta(\alpha)=\frac{\rho_{\Ra f}(u)}{\Gamma(\alpha+1)^{1/\alpha}}\]
    for $\alpha\in(-1,0)$ and $\alpha>0$. Therefore Lemma \ref{HL} implies that $\rho_{\rR_0 f}(u)=\lim_{\alpha\rightarrow 0}\rho_{\Ra f}(u)$ and
    \[\log\rho_{\rR_0f}(u)=-\gamma+\int_0^{\infty}\frac{1}{t}\Big(\frac{1}{\|f\|_2^2}\int_{\rn}f(x)f(x+tu)dx-e^{-t}\Big)dt.\]

   In our setting, we aim to extend the continuity of $\rho_{\Ra f}$ with respect to $\alpha$ to more general classes of functions. However, it is necessary to first analyze whether $\log\rho_{\rR_0 f}(u)$ is well-defined before moving forward.

    For non-negative, non-zero $f\in L^2(\rn)$ and fixed $u\in\sn$, we still use the expression for $g$ given in \eqref{g-def}, and have
    \begin{equation}\label{expand}
        \begin{aligned}
        \log\rho_{\rR_0f}(u)&=-\gamma+\int_0^{\infty}\frac{g(t)-e^{-t}}{t}dt\\&=C+\int_0^{1}\frac{g(t)-1}{t}dt+\int_1^{\infty}\frac{g(t)}{t}dt,
                           %&=-\gamma+\int_0^{1}\frac{g(t)-1}{t}dt+\int_0^{1}\frac{1-e^{-t}}{t}dt+\int_1^{\infty}\frac{g(t)}{t}dt+\int_1^{\infty}\frac{e^{-t}}{t}dt,
        \end{aligned}
    \end{equation}
    where $C$ is a constant given by
    \begin{equation}\label{C}
        C=-\gamma+\int_0^1\frac{1-e^{-t}}{t}dt+\int_1^{\infty}\frac{e^{-t}}{t}dt.
    \end{equation}
    By H\"older's inequality, it follows that $0\leq g(t)\leq 1$, and hence
    \[\int_0^1\frac{g(t)-1}{t}dt<0~~\text{and}~~\int_1^\infty\frac{g(t)}{t}dt>0.\]
    Therefore it is not clear whether $\log\rho_{\rR_0 f}(u)$ is well-defined, as the first part above may diverge to $-\infty$ while the other may diverge to $+\infty$. 
    
    %To solve this problem, we need more assumptions for $f$. 
    The following lemma says that $\log\rho_{\rR_0f}(u)$ is well-defined for all $u\in\sn$, if $f\in W^{\beta,2}(\rn)$ for some $\beta\in(0,\frac{1}{2})$.

    { \begin{lem}\label{0 well-defined}
    %\bigcup_{0<\beta<\frac{1}{2}}
        For $\beta\in(0,\frac{1}{2})$, let $f\in W^{\beta,2}(\rn)$ be non-negative, non-zero. Then $\log{\rho_{\rR_0f}}(u)$ is well-defined for all $u\in\sn$. If, in addition, $f\in L^1(\rn)$, then $\log\rho_{\rR_0 f}(u)$ is finite for almost all $u\in\sn$, and $\log\rho_{\rR_0 f}\in L^1(\sn)$.
    \end{lem}}

    \begin{proof}
    Let $g(t,u)$ be defined in \eqref{g-def}. By the representation formula \eqref{expand} for $\log\rho_{\rR_0f}(u)$, we first consider the integral from $0$ to $1$, where
        \begin{align*}
        %\Big(\frac{1}{\|f\|_2^2}\int_{\rn}f(x)f(x+tu)dx-1\Big)
            \int_0^1\frac{g(t,u)-1}{t}dt&=-\frac{1}{2\|f\|_2^2}\int_0^1\frac{1}{t}\int_{\rn}|f(x)-f(x+tu)|^2dxdt\\
            &\ge -\frac{1}{2\|f\|_2^2}\int_0^1 t^{-2\beta-1}\int_{\rn}|f(x)-f(x+tu)|^2dxdt\\
            &\ge -\frac{1}{2\|f\|_2^2}\int_0^{\infty} t^{-2\beta-1}\int_{\rn}|f(x)-f(x+tu)|^2dxdt=-\frac{1}{2\|f\|_2^2}\rho_{\Pi_2^{*,\beta}f}(u)^{-2\beta}.
        \end{align*}
        Since $f\in W^{\beta,2}(\rn)$, the set $\Pi_2^{*,\beta}f$ is a star body (see \cite[Proposition 4]{HL2}) and thus  $\rho_{\Pi_2^{*,\beta}f}(u)^{-2\beta}$ is uniformly bounded from below. This shows that $\log\rho_{\rR_0f}(u)$ is well-defined as an extended real number. Moreover, since $f\in W^{\beta,2}(\rn)$, we have
        \begin{equation}\label{finite-01}
            \begin{aligned}
                0\leq -\int_{\sn}\int_0^1\frac{g(t,u)-1}{t}dtdu&\leq \frac{1}{2\|f\|_2^2}\int_{\sn}\int_0^{\infty}t^{-2\beta-1}\int_{\rn}|f(x)-f(x+tu)|^2dxdtdu\\
                &=\frac{1}{2\|f\|_2^2}\di\frac{|f(x)-f(y)|^2}{|x-y|^{n+2\beta}}dxdy<\infty,
            \end{aligned}
        \end{equation}
        which shows that $\int_{\sn}\log\rho_{\rR_0f}(u)du$ is also well-defined, possibly taking value $+\infty$.

        If $f\in W^{\beta,2}(\rn)\cap L^1(\rn)$, Lemma \ref{lim} implies $f\in L^{\frac{2n}{n+\alpha^{\prime}}}(\rn)$ for some $\alpha^{\prime}\in(0,n)$. Then
        \begin{align*}
            \int_1^{\infty}\frac{g(t,u)}{t}dt&=\frac{1}{\|f\|_2^2}\int_1^{\infty}\frac{1}{t}\int_{\rn}f(x)f(x+tu)dxdt\\
            &\leq \frac{1}{\|f\|_2^2}\int_1^{\infty}t^{\alpha^{\prime}-1}\int_{\rn}f(x)f(x+tu)dxdt\\
            &\leq \frac{1}{\|f\|_2^2}\int_0^{\infty}t^{\alpha^{\prime}-1}\int_{\rn}f(x)f(x+tu)dxdt=\frac{1}{\|f\|_2^2}\rho_{\rS_{\alpha^{\prime}}f}(u)^{\alpha^{\prime}}.
        \end{align*}
        The HLS inequality implies that $\rho_{\rS_{\alpha^{\prime}}f}(u)^{\alpha^\prime}$ is finite for almost all $u\in\sn$ and
        \begin{equation}\label{finite-100}
            \begin{aligned}
                0\leq \int_{\sn}\int_1^{\infty}\frac{g(t,u)}{t}dtdu&\leq \frac{1}{\|f\|_2^2}\int_{\sn}\int_0^{\infty}t^{\alpha^{\prime}-1}\int_{\rn}f(x)f(x+tu)dxdtdu\\
                &=\frac{1}{\|f\|_2^2}\di \frac{f(x)f(y)}{|x-y|^{n-\alpha^{\prime}}}dxdy \leq \gamma_{n,\alpha^{\prime}}{\|f\|_{\frac{2n}{n+\alpha^{\prime}}}^2}{\|f\|_2^{-2}}<\infty.
            \end{aligned}
        \end{equation}
        Combining \eqref{expand}, \eqref{finite-01} and \eqref{finite-100}, we obtain that $\log\rho_{\rR_0f}\in L^1(\sn)$.
    \end{proof}   

    %\begin{rem}
    %    If we assume $f$ to be compactly supported, the integral on $(1,\infty)$ is finite
    %\end{rem}
    
    The following basic lemma describes the inclusion property between fractional $L^2$ Sobolev spaces, which is also presented in \cite{NPV}. For completeness, we provide a proof here.

    {\begin{lem}\label{frac l2}
        Let $f$ be non-negative. If $f\in W^{\beta,2}(\rn)$ for some $\beta\in(0,\frac{1}{2})$, then
        \[\int_{\rn}\int_{\rn}\frac{|f(x)-f(y)|^2}{|x-y|^{n-2\alpha}}dxdy<\infty,\]
        for $-\beta<\alpha<0$.
        %and therefore $\rho_{\Pi_2^{*,-\alpha/2}f}(u)$ is well-defined for almost all $u\in\sn$.
    \end{lem}}

    \begin{proof}
        By the fractional Laplacian operator introduced in Section 2, \eqref{pv} and the Parseval's formula, we can rewrite the double integral as
        \begin{equation}\label{fl}
        \begin{aligned}
            \di\frac{|f(x)-f(y)|^2}{|x-y|^{n-2\alpha}}dxdy&=2\di\frac{f(x)(f(x)-f(y))}{|x-y|^{n-2\alpha}}dxdy\\&=\pi^{\frac{n}{2}-2\alpha}\frac{2|\Gamma(\alpha)|}{\Gamma(\frac{n}{2}-\alpha)}\int_{\rn}|\xi|^{-2\alpha}|\hat{f}(\xi)|^2d\xi,
        \end{aligned}
        \end{equation}
        where 
        \[\int_{\rn}|\xi|^{-2\alpha}|\widehat{f}(\xi)|^2d\xi\leq \int_{B^n}|\widehat{f}(\xi)|^2d\xi+\int_{\rn\backslash B^n}|\xi|^{-2\alpha}|\widehat{f}(\xi)|^2d\xi.\]

        Since $f\in L^2(\rn)$,
        \[\int_{B^n}|\widehat{f}(\xi)|^2d\xi\leq \int_{\rn}|\widehat{f}(\xi)|^2d\xi=\|f\|_2^2.\]
        For $\xi\in\rn\backslash B^n$ and $\alpha\in(-\beta,0)$,
        \[\int_{\rn\backslash B^n}|\xi|^{-2\alpha}|\widehat{f}(\xi)|^2d\xi\leq\int_{\rn}|\xi|^{2\beta}|\widehat{f}(\xi)|^2d\xi,\]
        we then conclude the proof by \eqref{fl} and the assumption that $f\in W^{\beta,2}(\rn)$.
    \end{proof}

    We establish the following result on the continuity of $\rho_{\Ra f}(u)$ with respect to $\alpha$.

{ \begin{lem}\label{conti for cs}
    Let $u\in\sn$ be fixed and $\beta\in(0,\frac{1}{2})$. Suppose $f\in W^{\beta, 2}(\rn)$ is non-negative and compactly supported. Then $\rho_{\Ra f}(u)$ is continuous with respect to $\alpha$ on $(-2\beta,\infty)$. 
    %and
    %\[\lim_{\alpha\rightarrow 0^+}\rho_{\Ra f}(u)=\rho_{\rR_0 f}(u).\]
    
    %Additionally, if $f\in W^{\frac{1}{2},2}(\rn)$, after removing a set of measure zero on $\sn$, $\rho_{\Ra f}(u)$ is continuous with respect to $\alpha>-1$.
\end{lem}}

\begin{proof}
    Suppose $(\alpha_j)$ is an arbitrary sequence which converges to $\alpha_0$. Note that, for $\lambda>0$ and $f_{\lambda}(x):=f(\lambda x)$, we have $\Ra f_{\lambda}=\lambda^{-1}\Ra f$. Hence, we may assume ${\rm supp}f\subset \frac{1}{2}B^n$ without loss of generality. 
        
    \noindent{Case 1:} Let $\alpha_0>0$.
    
    After removing finitely many terms, we may assume $\alpha_j>0$ for all $j\in\mathbb{N}$.
    Since ${\rm supp}f\subset \frac{1}{2}B^n$, we have
    \[\rho_{\rR_{\alpha_j} f}(u)^{\alpha_j}=\frac{\alpha_j}{\|f\|_2^2}\int_0^1 t^{\alpha_j-1}\int_{\rn} f(x)f(x+tu)dxdt.\]
    By H{\"o}lder's inequality, 
    \[\alpha_jt^{\alpha_j-1}\frac{1}{\|f\|_2^2}\int_{\rn}f(x)f(x+tu)dx\leq \alpha_j t^{\alpha_j-1},\]
    where we have
    \[\lim_{j\rightarrow\infty}\int_0^1\alpha_jt^{\alpha_j-1}dt=1=\int_0^1\alpha_0t^{\alpha_0-1}dt.\]
    Then the dominated convergence theorem implies the continuity.

    {\noindent{Case 2:} Let $-2\beta<\alpha_0<0$.

    Similarly, we assume $-2\beta<\alpha_j<0$ for all $j\in\mathbb{N}$. Since $f\in W^{\beta,2}(\rn)$, Lemma \ref{frac l2} implies that $f\in W^{-\frac{\alpha}{2},2}(\rn)$ for all $\alpha\in(-2\beta,0)$ and thus $\rho_{\Ra f}(u)$ is well-defined { for all $u\in\sn$}. 
    
    Since $f$ is compactly supported, 
    \begin{align*}
    -\frac{1}{\alpha_j}\rho_{\rR_{\alpha_j}f}(u)^{\alpha_j}&=\int_0^{\infty}t^{\alpha_j-1}\Big(1-\frac{1}{\|f\|_2^2}\int_{\rn}f(x)f(x+tu)dx\Big)dt\\
    &=\int_0^{1}t^{\alpha_j-1}\Big(1-\frac{1}{\|f\|_2^2}\int_{\rn}f(x)f(x+tu)dx\Big)dt-\frac{1}{\alpha_j}.
    %&=\frac{1}{2\|f\|_2^2}\int_0^ct^{\alpha_j-1}\int_{\rn}|f(x)-f(x+tu)|^2dxdt-\frac{c^{\alpha_j}}{2\alpha_j}\\
    %&=\frac{c^{\alpha_j}}{2\|f\|_2^2}\int_0^1t^{\alpha_j-1}\int_{\rn}|f(x)-f(x+ctu)|^2dxdt-\frac{c^{\alpha_j}}{2\alpha_j}.
    \end{align*}
    By taking a subsequence, we may assume that $(\alpha_{j})$ is a monotone increasing (or decreasing) subsequence and therefore $t^{\alpha_j-1}$ is monotone decreasing (or increasing) since $t\in(0,1)$.  By the monotone convergence theorem, we have
    \[\lim_{j\rightarrow\infty}-\frac{1}{\alpha_{j}}\rho_{\rR_{\alpha_{j}} f}(u)^{\alpha_{j}}=\int_0^{1}t^{\alpha_0-1}\Big(1-\frac{1}{\|f\|_2^2}\int_{\rn}f(x)f(x+tu)dx\Big)dt-\frac{1}{\alpha_0},\]
    %=-\frac{1}{\alpha_0}\rho_{\rR_{\alpha_0}f}(u)^{\alpha_0}\]
    %\[\lim_{k\rightarrow\infty}-\frac{1}{\alpha_{j_k}}\rho_{\rR_{\alpha_{j_k}} f}(u)^{\alpha_{j_k}}=\frac{1}{2\|f\|_2^2}\int_0^1t^{\alpha_0-1}\int_{\rn}|f(x)-f(x+tu)|^2dxdt-\frac{1}{2\alpha_0},\]
    which yields the continuity of $\rho_{\Ra f}(u)$.}

    \noindent{Case 3:} 
    Let $\alpha_0=0$.
    
    We first consider the right-hand limit and thus assume $\alpha_j>0$ for all $j\in\mathbb{N}$. Note that
    \[\rho_{\Rf}(u)=(\Gamma(\alpha+1))^{\frac{1}{\alpha}}\left(1+\alpha\cdot\frac{\frac{1}{\alpha}\rho_{\Rf}(u)^{\alpha}-\Gamma(\alpha)}{\Gamma(\alpha+1)}\right)^{\frac{1}{\alpha}}.\]
    Since $\frac{1}{\alpha}\rho_{\Rf}(u)^{\alpha}-\Gamma(\alpha)$ is continuous with respect to $\alpha>0$, by applying $\lim_{\alpha\rightarrow 0^+}(1+\alpha)^{\frac{1}{\alpha}}=e$ and $\lim_{\alpha\rightarrow 0^+}(\Gamma(\alpha+1))^{\frac{1}{\alpha}}=e^{-\gamma}$, it suffices to prove
    \begin{equation}\label{lim r0}
    \lim_{j\rightarrow \infty}\left(\frac{1}{\alpha_j}\rho_{\rR_{\alpha_j}f}(u)-\Gamma(\alpha_j)\right)=\gamma+\log\rho_{\rR_0 f}(u).
    \end{equation}

    Let $g$ be defined in \eqref{g-def}. By the definition of $\Rf$, 
    \begin{align*}
    \frac{1}{\alpha_j}&\rho_{\rR_{\alpha_j}f}(u)^{\alpha_j}-\Gamma(\alpha_j)\\
    &=\int_0^1t^{\alpha_j-1}\left(g(t)-1\right)dt+\int_0^1t^{\alpha_j-1}(1-e^{-t})dt+\int_1^{\infty}t^{\alpha_j-1}g(t)dt-\int_1^{\infty}t^{\alpha_j-1}e^{-t}dt.
    \end{align*}
    Similarly, we may assume that $(\alpha_j)$ is a monotone sequence and $\alpha_j<1$ for $j$ large enough. Then an application of the monotone convergence theorem implies \eqref{lim r0}.
    
    %Since $g(t)\leq g(0)$ and $\frac{1-e^{-t}}{t}$ is integrable on $(0,c)$, by the dominated convergence theorem, we obtain
    %\[\lim_{j\rightarrow\infty}\int_0^ct^{\alpha_j-1}\left(\frac{g(t)}{g(0)}-e^{-t}\right)dt=\int_0^c \frac{1}{t}\left(\frac{g(t)}{g(0)}-e^{-t}\right)dt.\]

    %For the second integral on $(c,\infty)$,  there is $\alpha_j<1$ when $j$ is large enough and hence 
    %\[t^{\alpha_j-1}e^{-t}\leq c^{\alpha_j-1}e^{-t}\leq\max\{c^{-1}, 1\}e^{-t}.\] 
    %Since $e^{-t}$ is integrable on $(c,\infty)$, we obtain \eqref{lim r0} by the dominated convergence theorem.

    For the left-hand limit, we assume $\alpha_j\in(-2\beta,0)$ for all $j\in\mathbb{N}$. With the same notation and {the analytic continuation formula of the gamma function from \eqref{gn}}, we have
    %\begin{align*}
    %\frac{1}{\alpha_j}\rho_{\rR_{\alpha_j}f}(u)^{\alpha_j}-\Gamma(\alpha_j)&=\int_0^{c}t^{\alpha_j-1}\left(\frac{g(t)}{g(0)}-e^{-t}\right)dt+\int_c^{\infty}t^{\alpha_j-1}e^{-t}dt.
    %\end{align*}
    \[\frac{1}{\alpha_j}\rho_{\rR_{\alpha_j}f}(u)^{\alpha_j}-\Gamma(\alpha_j)=\int_0^1t^{\alpha_j-1}(g(t)-1)dt+\int_0^1t^{\alpha_j-1}(1-e^{-t})dt+\int_1^{\infty}t^{\alpha_j-1}e^{-t}dt.\]
    The same argument as in the right-hand limit case yields \eqref{lim r0}, which completes the proof. 
    %For the first integral, since $g(t)\leq g(0)$ and $1-e^{-t}\leq t$ for all $t\in\tr$,
    %\[t^{\alpha_j-1}\left(\frac{g(t)}{g(0)}-e^{-t}\right)\leq t^{\alpha_j},\]
    %where
    %\[\lim_{j\rightarrow\infty}\int_0^ct^{\alpha_j}dt=c=\int_0^c\lim_{j\rightarrow \infty}t^{\alpha_j}dt.\]
    %Then by the variant of dominated convergence theorem, we obtain
    %\[\lim_{j\rightarrow\infty}\int_0^ct^{\alpha_j-1}\left(\frac{g(t)}{g(0)}-e^{-t}\right)dt=\int_0^c\frac{1}{t}\left(\frac{g(t)}{g(0)}-e^{-t}\right)dt.\]
    %For the second integral, since $t>c$ and $-1<\alpha_j<0$,
    %\[t^{\alpha_j-1}e^{-t}\leq c^{\alpha_j-1}e^{-t}\leq \max\{c^{-2},1\}e^{-t},\]
    %where $e^{-t}$ is integrable on $(c,\infty)$. We then conclude the proof by the dominated convergence theorem.
\end{proof}

By Case 1 and Case 2 in the proof above, for non-negative, non-zero $f\in W^{\beta,2}(\rn)$ with compact support, $\rho_{\Ra f}(u)$ is uniformly bounded for $\alpha\in (-2\beta,0)\cup (0,\infty)$ and $u\in\sn$. By Lemma \ref{conti for cs},  we can directly get the following corollary. 

\begin{cor}\label{conti vol}
    For $\beta\in(0,\frac{1}{2})$, let $f\in W^{\beta,2}(\rn)$ be non-zero, non-negative and compactly supported. Then $|\Ra f|$ is positive and continuous with respect to $\alpha$ on $(-2\beta,\infty)$.
\end{cor}

We now move from compactly supported functions to general ones. For later use in the next lemma,  we define, for $u\in\sn$,
\begin{equation}\label{qu}
    \mathcal{E}_u(f)=\int_0^1\frac{1}{t}\int_{\rn}|f(x)-f(x+tu)|^2dxdt.
\end{equation}
We first record a simple observation: $\mathcal{E}_u^{1/2}$ satisfies the triangle inequality, in the sense that 
\begin{equation}\label{triangle}
    \mathcal{E}_u(f+h)^{1/2}\leq \mathcal{E}_u(f)^{1/2}+\mathcal{E}_u(h)^{1/2},
\end{equation}
whenever the quantities involved are finite. Indeed,
\begin{equation*}
    \begin{aligned}
        |(f(x)-f(x+tu))+(h(x)-h(x+tu))|^2&\leq |f(x)-f(x+tu)|^2+|h(x)-h(x+tu)|^2\\
        &+2|f(x)-f(x+tu)|\cdot|h(x)-h(x+tu)|.
    \end{aligned}
\end{equation*}
By the Cauchy--Schwarz inequality with respect to the measure
\(\frac{1}{t}dxdt\) on \(\rn\times(0,1)\), we have
\begin{equation*}
    \begin{aligned}
        \int_0^1\frac{1}{t}\int_{\rn}|f(x)-f(x+tu)&|\cdot|h(x)-h(x+tu)|dxdt \leq \mathcal{E}_u(f)^{1/2}\mathcal{E}_u(h)^{1/2},
    \end{aligned}
\end{equation*}
which implies the triangle inequality \eqref{triangle}.

\begin{lem}\label{mono conv}
Let $\beta\in(0,1)$ and $f\in W^{\beta,2}(\rn)$ be non-negative,  nonzero. 
Let $\eta\in C_c^\infty(\rn)$ be a cut-off function such that
$0\le \eta\le1$, with $\eta=1$ on $B^n$ and $\eta=0$ outside $2B^n$. Set
\[\eta_j(x)=\eta(x/j),\qquad f_j=\eta_j f .\]
Then $f_j\in W^{\beta, 2}(\rn)\cap L^1(\rn)$, and
\begin{equation}\label{pointwise}
    \lim_{j\to\infty}\log\rho_{\rR_0 f_j}(u)=\log\rho_{\rR_0 f}(u)
\end{equation}
for all $u\in\sn$. Moreover,
\begin{equation}\label{mono2}
    \lim_{j\to\infty}\int_{\sn}\log\rho_{\rR_0 f_j}(u)du=\int_{\sn}\log\rho_{\rR_0 f}(u)du .
\end{equation}
\end{lem}

{\begin{proof}
Since $\eta_j\in C_c^\infty(\rn)$ and
$f\in W^{\beta,2}(\rn)$, the multiplication property in
\cite[Theorem 6.23]{Leoni} gives $f_j\in W^{\beta,2}(\rn)$. Moreover, $f_j$ has compact support, which yields $f_j\in L^1(\rn)$.

As $f_j\leq f$ and $f_j\to f$ pointwise, the dominated convergence theorem gives $\|f_j\|_2\to\|f\|_2$. Moreover, $\|f\|_2>0$ since $f$ is non-zero. Let $g(t,u)$ be defined in \eqref{g-def}, namely
\[g(t,u)=\frac{1}{\|f\|_2^2}\int_{\mathbb R^n}f(x)f(x+tu)dx,~~u\in\sn.\]
We define $g_j(t, u)$ analogously, with $f$ replaced by $f_j$. By \eqref{expand} and Lemma \ref{0 well-defined},
\[\log\rho_{\rR_0 f_j}(u)=\int_0^1\frac{g_j(t,u)-1}{t}dt+\int_1^\infty \frac{g_j(t,u)}{t}dt+C\]
is finite almost everywhere and integrable over $\sn$, where $C$ is a constant defined in \eqref{C}.

We first consider the integral over $(1,\infty)$. Since $0\leq \eta_j(x)\leq 1$ and $\eta_j(x)=1$ for $x\in jB^n$,
\begin{equation}\label{b}
    \bar f_j(x)\bar f_j(x+tu)\le f_j(x)f_j(x+tu)\le f(x)f(x+tu),
\end{equation}
where $\bar{f}_j(x)=f(x)\chi_{jB^n}(x)$. Applying the monotone convergence theorem twice to $\bar f_j$, 
\[\lim_{j\to\infty}\int_1^\infty\frac{1}{t}\int_{\rn}\bar f_j(x)\bar f_j(x+tu)dxdt=\int_1^\infty\frac{1}{t}\int_{\rn}f(x)f(x+tu)dxdt=\|f\|_2^2\int_1^\infty\frac{g(t,u)}{t}dt.\]
Together with \eqref{b}, we have
\[\lim_{j\to\infty}\int_1^{\infty}\frac{1}{t}\int_{\rn}f_j(x)f_j(x+tu)dxdt=\int_1^\infty \frac{1}{t}\int_{\rn}f(x)f(x+tu)dxdt.\]
Since $\|f_j\|_2\rightarrow \|f\|_2$, we conclude that
\begin{equation}\label{01-point}
    \lim_{j\to\infty}\int_1^\infty\frac{g_j(t,u)}{t}dt=\int_1^\infty \frac{g(t,u)}{t}dt.
\end{equation}

Moreover, applying the monotone convergence theorem again, we have
\begin{equation*}
    \lim_{j\to\infty}\int_{\sn}\int_1^{\infty}\frac{1}{t}\int_{\rn}\bar f_j(x)\bar f_j(x+tu)dxdtdu=\int_{\sn}\int_1^\infty\frac{1}{t}\int_{\rn}f(x)f(x+tu)dxdtdu.
\end{equation*}
By the same argument, using \eqref{b} and the convergence $\|f_j\|_2\to\|f\|_2$, we obtain that
\begin{equation}\label{01-int}
    \lim_{j\to\infty}\int_{\sn}\int_1^\infty\frac{g_j(t,u)}{t}dtdu=\int_{\sn}\int_1^\infty \frac{g(t,u)}{t}dtdu.
\end{equation}

We next consider the integral over $(0,1)$. Let $\mathcal{E}_u(f)$ be defined in \eqref{qu}. Since $f \in W^{\beta,2}(\rn)$, by the first step in the proof of Lemma \ref{0 well-defined}, we have, for all $u\in\sn$, 
\[\int_0^1\frac{g(t,u)-1}{t}dt=-\frac{1}{2\|f\|_2^2}\int_0^1\frac{1}{t}\int_{\rn}|f(x)-f(x+tu)|^2dxdt=-\frac{1}{2\|f\|_2^2}\mathcal{E}_u(f)>-\infty.\]
%and
%\begin{equation}\label{4.3-step1}
%\mathcal{E}_u(f)\leq \int_0^{\infty}t^{-2\beta-1}\int_{\rn}|f(x)-f(x+tu)|^2dxdt<\infty.    
%\end{equation}
Since $\|f_j\|_2\to\|f\|_2$, it remains to prove $\lim_{j\rightarrow\infty}\mathcal{E}_u(f_j)=\mathcal{E}_u(f)$. By  the triangle inequality \eqref{triangle}, we have
\begin{equation}\label{triangle-extend}
    \begin{aligned}
        |\mathcal{E}_u(f_j)-\mathcal{E}_u(f)|&= |\mathcal{E}_u(f_j)^{1/2}-\mathcal{E}_u(f)^{1/2}|\cdot |\mathcal{E}_u(f_j)^{1/2}+\mathcal{E}_u(f)^{1/2}|\\
        &\leq \mathcal{E}_u(f_j-f)^{1/2}\Big(\mathcal{E}_u(f_j-f)^{1/2}+2\mathcal{E}_u(f)^{1/2}\Big)\\&=\mathcal{E}_u(f_j-f)+2\mathcal{E}_u(f)^{1/2}\mathcal{E}_u(f_j-f)^{1/2}
    \end{aligned}
\end{equation}
Then it suffices to show $\mathcal{E}_u(f_j-f)\to0$.

Let $h_j=f_j-f$. Then
\[
\begin{aligned}
h_j(x)-h_j(x+tu)
=(\eta_j(x)-1)\big(f(x)-f(x+tu)\big)+ f(x+tu)\big(\eta_j(x)-\eta_j(x+tu)\big).
\end{aligned}
\]
The inequality $(a+b)^2\leq 2a^2+2b^2$ implies that
\[\mathcal{E}_u(h_j)\le 2{\mathrm I}_j(u)+2{\mathrm{II}}_j(u),\]
where
\[\mathrm I_j(u)=\int_0^1\frac{1}{t}\int_{\rn}|\eta_j(x)-1|^2|f(x)-f(x+tu)|^2dxdt\]
and
\[
\mathrm{II}_j(u)=
\int_0^1\frac{1}{t}\int_{\rn}
|f(x+tu)|^2|\eta_j(x)-\eta_j(x+tu)|^2dxdt .
\]

Since $|\eta_j-1|\rightarrow0$ and $|\eta_j-1|\leq 1$, we have
\[\frac{1}{t}|\eta_j(x)-1|^2|f(x)-f(x+tu)|^2\leq \frac{1}{t}|f(x)-f(x+tu)|^2,\]
where the right-hand side is integrable and its integral equals $\mathcal{E}_u(f) < \infty$. Therefore, by the dominated convergence theorem, $\mathrm I_j(u)\to0$ for all $u$. 

On the other hand, since $\eta\in C_c^{\infty}(\rn)$, there is a constant $M$ such that $\|\nabla\eta_j\|_{\infty}\le \frac{M}{j}$. Then
\[|\eta_j(x)-\eta_j(x+tu)|
\le \frac{M}{j}t .\]
Therefore
\[\mathrm{II}_j(u)\le\frac{M^2}{j^2}\int_0^1 t\int_{\mathbb R^n}|f(x+tu)|^2dxdt=\frac{M^2}{2j^2}\|f\|_2^2\to0.\]
Thus $\mathcal{E}_u(f_j-f)=\mathcal{E}_u(h_j)\to0$. Combining this with
\eqref{01-point}, we obtain \eqref{pointwise}.

To obtain \eqref{mono2}, it remains to prove
\begin{equation*}
    \lim_{j\to\infty}\int_{\sn}\mathcal{E}_u(f_j)du=\int_{\sn}\mathcal{E}_u(f)du,
\end{equation*}
since $\|f_j\|_2\rightarrow\|f\|_2$. By \eqref{triangle-extend}, we have
\begin{equation}\label{step-triangle}
    \begin{aligned}
        |\mathcal{E}_u(f_j)-\mathcal{E}_u(f)|\leq \mathcal{E}_u(f_j-f)+2\mathcal{E}_u(f)^{1/2}\mathcal{E}_u(f_j-f)^{1/2}.
    \end{aligned}
\end{equation}
H\"older's inequality implies that
\begin{equation}\label{Holder}
    \int_{\sn}\mathcal{E}_u(f)^{1/2}\mathcal{E}_u(f_j-f)^{1/2}du\leq \Big(\int_{\sn}\mathcal{E}_u(f)du\Big)^{1/2}\Big(\int_{\sn} \mathcal{E}_u(f_j-f)du\Big)^{1/2}.
\end{equation}
By formula \eqref{finite-01}, $\int_{\sn}\mathcal{E}_u(f)du$ is finite. Together with \eqref{step-triangle} and \eqref{Holder}, it suffices to prove $\int_{\sn}\mathcal{E}_u(f_j-f)du\to 0$.

By the same estimates for $\mathrm I_j(u)$ obtained above, we have
\[\mathrm{I}_j(u)\rightarrow 0,\quad\quad \mathrm I_j(u)\leq \mathcal{E}_u(f).\]
Then the dominated convergence theorem implies that $\int_{\sn}\mathrm{I}_j(u)du\to0$. Moreover, the estimate $\mathrm{II}_j(u)\leq \frac{M^2}{2j^2}\|f\|_2^2$ implies that $\int_{\sn}\mathrm{II}_j(u)du\rightarrow0$. Therefore
\[\int_{\sn}\mathcal{E}_u(f_j-f)du\leq 2\int_{\sn}\mathrm{I}_j(u)du+2\int_{\sn}\mathrm{II}_j(u)du\to 0.\]
Combining this convergence with \eqref{01-int}, we obtain \eqref{mono2}.
\end{proof}}

We also need the basic result: for continuous $F:[0,\infty)\rightarrow (0,\infty)$,
\begin{equation}\label{fund}
\lim_{\alpha\rightarrow 0^+}\frac{F(\alpha)^{\alpha}-1}{\alpha}=\log F(0).    
\end{equation}
%({\cb On page 47 line 15, (4.9) can be found on page 2023 (3.2) of MR3833785})
%MR3833785 - A mixed volumetry for the anisotropic logarithmic potential Hou, S.; Xiao, J. J. Geom. Anal. 28 (2018), no. 3, 2028-2049. (Reviewer:Ryabogin, Dmitry)

We are now in the position to prove the main result in this chapter, the  limiting case ($\alpha\rightarrow 0^+$) of the affine HLS inequalities. We first provide  another form of Beckner's logarithmic Sobolev inequality.

Note that by Lemma \ref{0 well-defined}, the integral of $\log \rho_{\rR_0f}$ is finite for non-negative, non-zero $f\in W^{\beta,2}(\rn)\cap L^1(\rn)$. For the integral of $|\hat{f}(\xi)|^2\log |\xi|$, by Jensen's inequality, we have
\[\int_{\rn}|\hat{f}(\xi)|^2\log|\xi|d\xi=\frac{1}{\alpha}\int_{\rn}|\hat{f}(\xi)|^2\log|\xi|^{\alpha}d\xi\leq \frac{\|f\|_2^2}{\alpha}\log \Big(\frac{1}{\|f\|_2^2}\int_{\rn}|\hat{f}(\xi)|^2|\xi|^{\alpha}d\xi\Big),\]
where $\alpha\in(0,2\beta)$. Since $f\in W^{\beta,2}(\rn)$, we then have $\int_{\rn}|\hat{f}(\xi)|^2|\xi|^{\alpha}d\xi<+\infty$.

By Lemma \ref{lim}, $f\in L^{\frac{2n}{n+\alpha^{\prime}}}(\rn)$ for some $\alpha^{\prime}\in(0,n)$. The HLS inequality and \eqref{riesz potential} imply that $|\hat{f}(\xi)|^2|\xi|^{-\alpha^{\prime}}\in L^1(\rn)$. It follows from Jensen's inequality that
\[-\infty<-\frac{\|f\|_2^2}{\alpha^{\prime}}\log \Big(\frac{1}{ \|f\|_2^2}\int_{\rn}|\hat f(\xi)|^2|\xi|^{-\alpha^{\prime}}d\xi\Big)\leq \int_{\rn}|\hat{f}(\xi)|^2\log|\xi|d\xi.\]
Thus the integral of $|\hat{f}(\xi)|^2\log |\xi|$ is finite. Then the following result implies another form of Beckner's logarithmic Sobolev inequality.

\begin{prop}\label{eqiv}
    For $\beta\in (0,\frac{1}{2})$, let $f\in W^{\beta,2}(\rn)\cap L^1(\rn)$ be non-negative,  non-zero. Then
    \[-\frac{1}{n\omega_n}\int_{\sn}\log \rho_{\rR_0 f}(u)du=\frac{1}{\|f\|_2^2}\int_{\rn}|\hat{f}(\xi)|^2\log|\xi|d\xi+\frac{1}{2}\Big(\psi(\frac{n}{2})-\psi(1))\Big).\]
\end{prop}

\begin{proof}
    By the fractional Laplacian and the definition of $\Ra f$, 
    \begin{equation}\label{lim2}
        \alpha\di\frac{f(x)f(y)}{|x-y|^{n-\alpha}}dxdy=\frac{\alpha\pi^{\frac{n}{2}-\alpha}\Gamma(\frac{\alpha}{2})}{\Gamma(\frac{n-\alpha}{2})}\int_{\rn}|\xi|^{-\alpha}|\hat{f}(\xi)|^2d\xi=\|f\|_2^2\int_{\sn}\rho_{\Ra f}(u)^{\alpha}du.
    \end{equation}
    Let $c_{n,\alpha}$ be the constant coefficient in the second term, where $\lim_{\alpha\rightarrow 0}c_{n,\alpha}=n\omega_n$. We have
    \begin{equation}\label{eq}
        c_{n,\alpha}\int_{\rn}|\hat{f}(\xi)|^2\cdot \frac{|\xi|^{-\alpha}-1}{\alpha}d\xi+\|f\|_2^2\cdot \frac{c_{n,\alpha}-n\omega_n}{\alpha}=\|f\|_2^2\int_{\sn}\frac{\rho_{\Ra f}(u)^{\alpha}-1}{\alpha} du
    \end{equation}
    %{\cb Since $f\in W^{\beta,2}(\rn)\cap L^1(\rn)$, Lemma \ref{lim} and the HLS inequality imply that $|\hat{f}(\xi)|^2|\xi|^{-\alpha_0}$ is integrable for some $\alpha_0\in(0,n)$.}
    
    We first consider $f$ with compact support. By Case 1 in the proof of Lemma \ref{conti for cs}, there is a uniform constant $c_1>0$ such that $\rho_{\Ra f}(u)\leq c_1$ for all $u\in\sn$. Moreover, since $e^{t}-1\ge t$ for all $t\in\tr$, we have
    \[\log\rho_{\Ra f}(u)\leq \frac{\rho_{\Ra f}(u)^{\alpha}-1}\alpha\leq \frac{c_1^{\alpha}-1}{\alpha}.\]
    By Lemma \ref{0 well-defined}, there is a uniform lower bound for $\log\rho_{\rR_0f}$. {By the proof of Lemma \ref{conti for cs}, $\lim_{\alpha\rightarrow 0^+}\rho_{\Ra f}(u)=\rho_{\rR_0f}(u)$ uniformly for $u\in \sn$}. For sufficiently small $\alpha$, there is $c_2$ such that $\log\rho_{\Ra f}(u)\ge c_2$ for all $u\in\sn$ and $\alpha>0$. 
    Then the dominated convergence theorem, together with \eqref{fund}, implies that
    \[\lim_{\alpha\rightarrow 0^+}\int_{\sn}\frac{\rho_{\Ra f}(u)^{\alpha}-1}{\alpha} du=\int_{\sn}\log\rho_{\rR_0 f}(u)du.\]

    For the left-hand side of \eqref{eq}, note that $\frac{d}{d\alpha}|_{\alpha=0}c_{n,\alpha}=\frac{n\omega_n}{2}(\psi(1)-\psi(\frac{n}{2}))$, and
    \[\int_{\rn}|\hat{f}(\xi)|^2\cdot\frac{|\xi|^{-\alpha}-1}{\alpha}d\xi=\int_{B^n}|\hat{f}(\xi)|^2\cdot\frac{|\xi|^{-\alpha}-1}{\alpha}d\xi+\int_{\rn\backslash B^n}|\hat{f}(\xi)|^2\cdot\frac{|\xi|^{-\alpha}-1}{\alpha}d\xi.\]    
    where $\frac{|\xi|^{-\alpha}-1}{\alpha}$ is monotone decreasing with respect to $\alpha$. Therefore, for sufficiently small $\alpha>0$,
    \[\frac{|\xi|^{-\alpha^{\prime}}-1}{\alpha^{\prime}}\leq \frac{|\xi|^{-\alpha}-1}{\alpha}\leq \log|\xi|,~~\xi\in B^n,\]
    with $\alpha^{\prime}\in(0,n)$ such that $f\in L^{\frac{2n}{n+\alpha^{\prime}}}(\rn)$, and
    \[0<\frac{|\xi|^{-\alpha}-1}{\alpha}\leq \log|\xi|,~~\xi\in \rn\backslash B^n.\] 
    Hence the dominated convergence theorem implies that
    \begin{equation}\label{lim1}
        \lim_{\alpha\rightarrow 0^+}\int_{\rn}|\hat{f}(\xi)|^2\cdot\frac{|\xi|^{-\alpha}-1}{\alpha}d\xi=-\int_{\rn}|\hat{f}(\xi)|^2\log|\xi|d\xi,
    \end{equation}
    which yields
    \[\frac{1}{2}\Big(\psi(\frac{n}{2})-\psi(1))\Big)+\frac{1}{\|f\|_2^2}\int_{\rn}|\hat{f}(\xi)|^2\log|\xi|d\xi=-\frac{1}{n\omega_n}\int_{\sn}\log\rho_{\rR_0 f}(u)du.\]
    
    For general $f$, let $f_j$ be defined in Lemma \ref{mono conv}, which implies that $f_j\in W^{\beta,2}(\rn)\cap L^1(\rn)$. By Lemma \ref{mono conv} and the fact that $\|\hat{f_j}\|_2=\|f_j\|_2$, it suffices to prove
    \[\lim_{i\rightarrow\infty}\int_{\rn}|\hat{f_j}(\xi)|^2\log|\xi|d\xi=\int_{\rn}|\hat{f}(\xi)|^2\log|\xi|d\xi.\]

    Note that the convergence in \eqref{lim1} holds for general $f$ without compact support and therefore does not rely on ${\rm supp}f$. By \eqref{lim2} and \eqref{lim} and $\|\hat{f_j}\|_2=\|f_j\|_2$, for arbitrary $\epsilon>0$, 
    \begin{equation*}
        \begin{aligned}
            \frac{1}{c_{n,\alpha}}\di &\frac{f_j(x)f_j(y)}{|x-y|^{n-\alpha}}dxdy-\frac{\|f_j\|_2^2}{\alpha}-\epsilon\\
            &\leq -\int_{\rn}|\hat{f_j}(\xi)|^2\log|\xi|d\xi\leq \frac{1}{c_{n,\alpha}}\di \frac{f_j(x)f_j(y)}{|x-y|^{n-\alpha}}dxdy-\frac{\|f_j\|_2^2}{\alpha}+\epsilon
        \end{aligned}
    \end{equation*}
    holds for sufficiently small $\alpha$.    Letting $j\rightarrow \infty$ and applying \eqref{lim2} again, we obtain
    \begin{equation*}
    \begin{aligned}
        \int_{\rn}|\hat{f}(\xi)|^{2}\cdot\frac{|\xi|^{-\alpha}-1}{\alpha}d\xi-\epsilon\leq        -\lim_{j\rightarrow\infty}\int_{\rn}|\hat{f_j}(\xi)|^2\log |\xi|d\xi\leq \int_{\rn}|\hat{f}(\xi)|^{2}\cdot\frac{|\xi|^{-\alpha}-1}{\alpha}d\xi+\epsilon.
    \end{aligned}
    \end{equation*}
    We then conclude the proof.
\end{proof}

We are now in a position to prove the following Beckner-type logarithmic Sobolev inequality, with sharp constants and extremal functions to be characterized in Section 5.

\begin{thm}\label{main 2}
    For $\beta\in(0,\frac{1}{2})$, let $f\in W^{\beta,2}(\rn)$ be a non-negative function such that $\|f\|_2=1$ and  
 $\int_{\rn}f(x)^2\log f(x)dx$ is finite. Then
    \begin{align}\label{ineq 0}
    \gamma_0-\frac{2}{n}\int_{\rn}f(x)^2\log f(x)dx\ge\frac{1}{n}\log\Big(\frac{|{\rm R}_0f|}{\omega_n}\Big)\ge\frac{1}{n\omega_n}\int_{\sn}\log\rho_{{\rm R}_0f}(u)du,
    \end{align}
    where $\gamma_0$ is the derivative of $\frac{1}{n\omega_n}\alpha\oc$ at $\alpha=0$, and
    \[\gamma_0=-\frac{1}{2}\log\pi+\frac{1}{n}\log\frac{\Gamma(n)}{\Gamma(\frac{n}{2})}+\frac{1}{2}\left(\psi(1) -\psi(\frac{n}{2})\right).\]
\end{thm}

\begin{proof}
    The second inequality follows directly from Jensen's inequality. We focus on the first inequality.

    We first assume that $f$ is compactly supported. Then $f\in W^{\beta,2}(\rn)\cap L^1(\rn)$.  By Lemma \ref{lim}, we can subtract $n\omega_n$ from each term in the affine HLS inequalities and divide by $n\omega_n\alpha$. Then we consider the limit of each part as $\alpha$ goes to $0^+$. 

    For the first part we have
    \begin{align*}
    \lim_{\alpha\rightarrow 0^+}&\frac{\alpha\gamma_{n,\alpha}\|f\|^2_{\frac{2n}{n+\alpha}}-n\omega_n}{n\omega_n\alpha}\\&\qquad=\lim_{\alpha\rightarrow 0^+}\frac{1}{n\omega_n}\cdot\frac{\alpha\gamma_{n,\alpha}-n\omega_n}{\alpha}+\lim_{\alpha\rightarrow 0^+}\frac{\|f\|_{\frac{2n}{n+\alpha}}^2-1}{\alpha}\\
    &\qquad=\gamma_0+\lim_{\alpha\rightarrow 0^+}\left(\int_{\rn}f(x)^{\frac{2n}{n+\alpha}}dx\cdot \frac{(\int_{\rn}f(x)^{\frac{2n}{n+\alpha}}dx)^{\frac{\alpha}{n}}-1}{\alpha}+\frac{\int_{\rn}f(x)^{\frac{2n}{n+\alpha}}dx-1}{\alpha}\right)\\
    %&=\gamma_0+\lim_{\alpha\rightarrow 0^+}\left(\frac{\int_{\rn}f(x)^{\frac{2n}{n+\alpha}}dx-1}{\alpha}+\frac{\left(\int_{\rn}f(x)^{\frac{2n}{n+\alpha}}dx\right)^{\frac{\alpha}{n}}-1}{\alpha}\right)\\
    &\qquad=\gamma_0+\lim_{\alpha\rightarrow 0^+}\frac{\int_{\rn}f(x)^{\frac{2n}{n+\alpha}}dx-1}{\alpha},
    \end{align*}
    where the last step uses \eqref{fund}.
    
    Note that $\|f\|_2=1$. 
    %then it suffices to consider the following limit
    %\[\lim_{\alpha\rightarrow 0^+}\int_{\rn}\frac{f(x)^{\frac{2n}{n+\alpha}}-f(x)^2}{\alpha}dx.\]
    Since
    \[\frac{f(x)^{\frac{2n}{n+\alpha}}-f(x)^2}{\alpha}=-\frac{2}{n+\alpha}\cdot\frac{f(x)^{\frac{2n}{n+\alpha}}-f(x)^2}{\frac{2n}{n+\alpha}-2},\]
    and $(p-2)^{-1}(f^p-f^2)$ is monotone increasing when $p\rightarrow 2^-$, by the monotone convergence theorem, we have
    \[\lim_{\alpha\rightarrow 0^+}\int_{\rn}\frac{f(x)^{\frac{2n}{n+\alpha}}-f(x)^2}{\alpha}dx=-\frac{2}{n}\int_{\rn}f(x)^2\log f(x)dx,\]
    which yields the first part.
    %Hence we have
    %\[\lim_{\alpha\rightarrow 0^+}\frac{\alpha\gamma_{n,\alpha}\|f\|^2_{\frac{2n}{n+\alpha}}-n\omega_n}{n\omega_n\alpha}=\gamma_0-\frac{2}{n}\int_{\rn}f(x)^2\log f(x)dx\]

    Then we consider the second part. Recall that $\Rf=\Big(\frac{\alpha}{\|f\|_2^2}\Big)^{1/\alpha}\Sa=\alpha^{\frac{1}{\alpha}}\Sa$, and therefore 
    \[\alpha n\omega_n^{\frac{n-\alpha}{n}}|\Sa|^{\frac{\alpha}{n}}=n\omega_n^{\frac{n-\alpha}{n}}|\Rf|^{\frac{\alpha}{n}}.\]
    %Since
    %\[\lim_{\alpha\rightarrow 0^+}\frac{|\Rf|^{\frac{\alpha}{n}}-1}{\alpha}=\lim_{\alpha\rightarrow 0^+}\left(\frac{e^{\frac{\alpha}{n}\log |\Rf|}-1}{\frac{\alpha}{n}\log |\Rf|}\cdot\frac{\log |\Rf|}{n}\right)=\frac{\log |\rR_0 f|}{n},\]
    Applying Corollary \ref{conti vol} and \eqref{fund}, we conclude the proof by 
    \begin{align*}
    \lim_{\alpha\rightarrow 0^+}\frac{\alpha n\omega_n^{\frac{n-\alpha}{n}} |\Sa|^{\frac{\alpha}{n}}-n\omega_n}{n\omega_n\alpha}=\lim_{\alpha\rightarrow 0}\frac{n\omega_n^{\frac{n-\alpha}{n}}-n\omega_n}{n\omega_n\alpha}+\lim_{\alpha\rightarrow 0}\frac{|\Rf|^{\frac{\alpha}{n}}-1}{n\omega_n\alpha}=\frac{1}{n}\log\Big(\frac{|\rR_0 f|}{\omega_n}\Big).
    \end{align*}
    
    For general $f\in W^{\beta, 2}(\rn)$, let $f_j$ be defined in Lemma \ref{mono conv}. Then $f_j\in W^{\beta,2}(\rn)\cap L^1(\rn)$ and hence inequality \eqref{ineq 0} holds for $f_j/\|f_j\|_2$. Observing that replacing $f$ with $cf$ does not change $\rR_0 f$, we have
    \begin{equation}\label{j}
        \gamma_0+\frac{2}{n}\log\|f_j\|_2-\frac{2}{n\|f_j\|_2^2}\int_{\rn}f_j(x)^2\log f_j(x)dx\ge\frac{1}{n}\log\Big(\frac{|\rR_0 f_j|}{\omega_n}\Big)\ge\frac{1}{n\omega_n}\int_{\sn}\log\rho_{\rR_0 f_j}(u)du.
    \end{equation}
    %\[\gamma_0+\frac{2}{n}\log\|f_j\|_2-\frac{2}{n\|f_j\|_2^2}\int_{\rn}f_j(x)^2\log f_j(x)dx\ge\frac{1}{n}\log\Big(\frac{|\rR_0 f_j|}{\omega_n}\Big)\ge\frac{1}{n\omega_n}\int_{\sn}\log\rho_{\rR_0 f_j}(u)du.\]
    By Fatou's lemma, Lemma \ref{mono conv} and Jensen's inequality,
    \begin{equation*}
        \begin{aligned}
            \liminf_{j\rightarrow\infty}\frac{1}{n}\log\Big(\frac{|\rR_0 f_j|}{\omega_n}\Big)\ge \frac{1}{n}\log\Big(\frac{|\rR_0 f|}{\omega_n}\Big)&\ge \frac{1}{n\omega_n}\int_{\sn}\log\rho_{\rR_0f}(u)du.
        \end{aligned}
    \end{equation*}
    Then letting $j\rightarrow\infty$ in \eqref{j} completes the proof.
    \end{proof}

        By Lemma \ref{conti for cs} and the method of the proof above, we obtain the following Corollary, which shows that the inequality \eqref{ineq 0} is also a limiting case of affine fractional $L^2$ Sobolev inequalities. 

    \begin{cor}
        Let $\beta\in(0,\frac{1}{2})$. Suppose { $f\in W^{\beta,2}(\rn)$} is a non-negative function such that $\|f\|_2=1$ and  $\int_{\rn}f(x)^2\log f(x)dx$ is finite. Then \eqref{ineq 0} is the limiting case of the affine fractional $L^2$ Sobolev inequalities, as $\alpha\rightarrow 0^-$.
    \end{cor}

%{\cb    As previously discussed, fractional $L^2$ Sobolev inequalities are equivalent to HLS inequalities. We may view fractional $L^2$ Sobolev inequalities as the `negative' part of HLS inequalities. Then the inequality \eqref{ineq 0} is the connection point between them. }

\section{Equality cases of the affine logarithmic Sobolev inequality}

In this section, we will explore Beckner's brief remark on the optimizer of his logarithmic Sobolev inequality,
offering a more detailed explanation of his idea. Building on this foundation, we will proceed to characterize the equality cases in the affine version of the logarithmic Sobolev inequality.

The following theorem is a crucial result we need for the characterization. This theorem extends the result in \cite[Theorem 20]{HL1} by Haddad and Ludwig, which is for convex bodies $K$ and $E$ with $f=\chi_E$ in \eqref{rri v-log}.
\begin{thm}\label{main 5}
    Let $\beta\in(0,\frac{1}{2})$. Suppose $f\in W^{\beta,2}(\rn)$ is non-zero and non-negative and { $K$ is a star-shaped set with measurable radial function and $|K|>0$}. Then
    \begin{equation}\label{rri v-log}
        \tilde{V}_{\log}(K,\rR_0f)\leq \tilde{V}_{\log}(K^{\star},\rR_0 f^{\star}),
    \end{equation}
    provided both sides are finite. Equality holds if and only if 
    \[f(x)=f^{\star}(\phi^{-1}x-x_0)~~\text{and}~~K=\phi B,\]
    for some $\phi\in {\rm SL}(n)$, $x_0\in\rn$ and $B$ is a centered ball with $|B|=|K|$.
    %$K$ is a centered ellipsoid and $f$ is a translate of $f^{\star}\circ\phi^{-1}$ for some $\phi\in{\rm SL}(n)$.
\end{thm}

For the proof of Theorem \ref{main 5}, we first need the following results, where $f_+$ denotes the positive part of $f$ given by $f_+(x)=\max\{0,f(x)\}$.

%The following Lemmas provide the characterization of the optimizers of the second inequality.

%says the equality holds in the second inequality if and only if 
%$f=f^{\star}$ up to a volume-preserving transform.

\begin{lem}\label{P-S}
    Let $\beta\in(0,\frac{1}{2})$. Suppose $f\in W^{\beta,2}(\rn)$ is non-zero and non-negative and { $K\subset\rn$ is star-shaped with measurable radial function.} If
    \[\di(f(x)-f(y))_+^2\cdot\max\{0,\|x-y\|_K^{-n}-1\}dxdy<\infty,\]
    then
    \begin{equation}\label{in}
        \begin{aligned}
    \di(f(x)-f(y))_+^2&\cdot\max\{0,\|x-y\|_K^{-n}-1\}dxdy\\
    &\ge\di(f^{\star}(x)-f^{\star}(y))_+^2\cdot\max\{0,\|x-y\|_{K^{\star}}^{-n}-1\}dxdy.
    \end{aligned}
    \end{equation}
    There is equality in \eqref{in} if and only if 
    \[f(x)=f^{\star}(\phi^{-1}x-x_0)~~\text{and}~~K=\phi B,\]
    for some $\phi\in{\rm SL}(n)$, $x_0\in\rn$ and $B$ is a centered ball with $|B|=|K|$.
    %$K$ is a centered ellipsoid and $f$ is a translate of $f^{\star}\circ\phi^{-1}$ for some $\phi\in{\rm SL}(n)$.}
\end{lem}

\begin{proof}
    Since 
    $\|z\|_K^{-n}=\int_0^{\infty}\chi_{t^{-\frac{1}{n}}K}(z)dt$, the kernel function is
    \[\max\{0,\|z\|_K^{-n}-1\}=\int_1^{\infty}\chi_{t^{-\frac{1}{n}}K}(z)dt,\]
    and we obtain
    \begin{align*}
    J:=\di (f(x)-f(y))_+^2&\cdot\max\{0,\|x-y\|_K^{-n}-1\}dxdy\\
    &=\int_1^{\infty}\di (f(x)-f(y))_+^2\cdot\chi_{t^{-\frac{1}{n}}K}(x-y)dxdy.
    \end{align*}
    Note that
    \begin{equation}\label{id1}
        (f(x)-f(y))_+^2=2\int_0^{\infty}(f(x)-r)_+\chi_{\{f<r\}}(y)dr,
    \end{equation}
    and that
    \begin{equation}\label{id2}
        (f(x)-r)_+=\int_r^{\infty}\chi_{\{f\ge s\}}(x)ds.
    \end{equation}
    By Fubini's theorem, \eqref{id1}, \eqref{id2} and the fact that $\chi_{\{f<r\}}=1-\chi_{\{f\ge r\}}$, we can rewrite the integral $J$ as
    \begin{align*}
        J=2\int_1^{\infty}\int_0^{\infty}\Big(t^{-1}|K|&\int_{\rn}(f(x)-r)_+dx\\
        &-\int_r^{\infty}\di \chi_{\{f\ge s\}}(x)\chi_{t^{-\frac{1}{n}}K}(x-y)\chi_{\{f\ge r\}}(y)dxdyds\Big)drdt.
    \end{align*}

    By the Riesz rearrangement inequality, we have
    \begin{equation}\label{rri s}
        \begin{aligned}
        \di \chi_{\{f\ge s\}}\chi_{t^{-\frac{1}{n}}K}&(x-y)\chi_{\{f\ge r\}}(y)dxdy\\
        &\leq \di \chi_{\{f\ge s\}^{\star}}\chi_{t^{-\frac{1}{n}}K^{\star}}(x-y)\chi_{\{f\ge r\}^{\star}}(y)dxdy.
    \end{aligned}
    \end{equation}
   Since $\int_{\rn}(f(x)-r)_+dx<\infty$ and $t^{-1}|K|\int_{\rn}(f(x)-r)_+dx$ is invariant under Schwarz symmetrization, \eqref{rri s} yields \eqref{in}.

  If equality holds in \eqref{in}, for almost all $r>0$, we have $(r,s,t)$ satisfying
    \begin{align*}
        \di\chi_{\{f\ge s\}}(x)
        &\chi_{t^{-\frac{1}{n}}K}(x-y)\chi_{\{f\ge r\}}(y)dxdy\\
        &=\di\chi_{\{f^{\star}\ge s\}}(x)
        \chi_{t^{-\frac{1}{n}}K^{\star}}(x-y)\chi_{\{f^{\star}\ge r\}}(y)dxdy.
    \end{align*}
    for almost all $t>1$ and $s>r$. Then for such $(r,s)$ with $s>r$ and $t>1$ sufficiently large, the assumptions of Theorem \ref{RRI B2} are fulfilled and thus
    \[\{f\ge s\}=x_1+\alpha D,~~~~ t^{-\frac{1}{n}}K=x_2+\beta D,~~~~\{f\ge r\}=x_3+\gamma D,\]
    where $D$ is a centered ellipsoid, $x_1,x_2$ and $x_3=x_1+x_2\in\rn$. Since $K=t^{\frac{1}{n}}b+(|K|/|D|)^{1/n}D$, then $D$ is not dependent on $(r,s,t)$ and $x_2=0$. Hence $x_1=x_3$ is a constant vector, which concludes the proof.
\end{proof}

We can obtain the result for $(f(x)-f(y))_-$ similarly, where $f_-$ is the negative part of $f$ given by $f_-(x)=\max\{0,-f(x)\}$. Hence we have the following Corollary.

\begin{corollary}\label{P-S-C}
Let $\beta\in(0,\frac{1}{2})$. Suppose $f\in W^{\beta,2}(\rn)$ is non-zero and non-negative and { $K\subset\rn$ is star-shaped with measurable radial function.} If
    \[\di|f(x)-f(y)|^2\cdot\max\{0,\|x-y\|_K^{-n}-1\}dxdy<\infty,\]
    then
    \[\di|f(x)-f(y)|^2\cdot\{0,\|x-y\|_K^{-n}-1\}dxdy\ge\di|f^{\star}(x)-f^{\star}(y)|^2\cdot\{0,\|x-y\|_{K^{\star}}^{-n}-1\}dxdy,\]
    with equality if and only if
    \[f(x)=f^{\star}(\phi^{-1}x-x_0),~~\text{and}~~K=\phi B,\]
    for some $\phi\in{\rm SL}(n)$, $x_0\in\rn$ and $B$ is a centered ball with $|B|=|K|$.
\end{corollary}

\begin{proof}[Proof of Theorem 5.1]
    Without loss of generality, we assume $\|f\|_2=1$.  We keep using the expression of $g$ defined in \eqref{g-def}, but use $g(tu)$ instead of $g(t)$, where $u\in\sn$. Recall that
    \[\log\rho_{\rR_0 f}(u)=\int_0^1\frac{g(t,u)-1}{t}dt+\int_1^{\infty}\frac{g(t,u)}{t}dt+C.\]
    Then
    \begin{align*}
    n|K|\tilde{V}_{\log}(K,\rR_0 f)&=\int_{\sn}\rho_K(u)^n\log\Big(\frac{\rho_{\rR_0f}(u)}{\rho_K(u)}\Big)du\\
    &=\int_{\sn}\rho_K(u)^n\Big(\int_0^1\frac{g(t,u)-1}{t}dt+\int_1^{\infty}\frac{g(t,u)}{t}dt-\int_1^{\rho_K(u)}\frac{1}{t}dt\Big)du+C\\
    &=\int_{\sn}\rho_K(u)^n\Big(\int_0^{\rho_K(u)}\frac{g(t,u)-1}{t}dt+\int_{\rho_K(u)}^{\infty}\frac{g(t,u)}{t}dt\Big)du+C.
    \end{align*}
    By polar coordinate, we have
    \begin{align*}
        n|K|\tilde{V}_{\log}(K,\rR_0f)=&-\di \frac{f(x)(f(x)-f(y))\chi_{(0,1)}(\|x-y\|_K)}{\|x-y\|_K^n}dxdy\\
        &+\di \frac{f(x)f(y)\chi_{[1,\infty)}(\|x-y\|_K)}{\|x-y\|_K^n}dxdy.
    \end{align*}
    
    Observing that
    \[\chi_{(0,1)}(\|x-y\|_K)\cdot \|x-y\|_K^{-n}=\chi_{(0,1)}(\|x-y\|_K)+\max\{0,\|x-y\|_K^{-n}-1\},\]
    and
    \[\chi_{[1,\infty)}(\|x-y\|_K)\cdot \|x-y\|_K^{-n}=\min\{1,\|x-y\|_K^{-n}\}-\chi_{(0,1)}(\|x-y\|_K),\]
    we then have
    \begin{align*}
        n|K|\tilde{V}_{\log}(K,\rR_0f)=&-\frac{1}{2}\di |f(x)-f(y)|^2\cdot\max\{0,\|x-y\|_K^{-n}-1\}dxdy\\
        &+\di f(x)f(y)\cdot \min\{1,\|x-y\|_K^{-n}\}dxdy+|K|.
    \end{align*}
    
    Since $|K|=|K^{\star}|$, combined with Corollary \ref{P-S-C} and the Riesz rearrangement inequality, we have
    \[\tilde{V}_{\log}(K,\rR_0 f)\leq \tilde{V}_{\log}(K^{\star},\rR_0 f^{\star}),\]
    {with equality if and only if}
    \[f(x)=f^{\star}(\phi^{-1}x-x_0)~~\text{and}~~K=\phi B,\]
    where $\phi\in{\rm SL}(n)$, $x_0\in\rn$ and $B$ is a centered ball.
\end{proof}

By Theorem \ref{main 5}, we also obtain the following result, which will be used to characterize the optimizers of the affine logarithmic Sobolev inequality stated in Theorem \ref{main 2}.

\begin{corollary}\label{rri vol}
    Let $\beta\in(0,\frac{1}{2})$. Suppose $f\in W^{\beta,2}(\rn)$ is non-negative. Then
    \[|\rR_0 f|\leq |\rR_0 f^{\star}|,\]
    and equality holds if and only if $f$ is a translate of $f^{\star}\circ\phi^{-1}$ for some $\phi\in{\rm SL}(n)$.
    %\[f(x)=f^{\star}(\phi^{-1}x-x_0),\]
    %for some $\phi\in{\rm SL}(n)$ and $x_0\in\rn$.
\end{corollary}

\begin{proof}
    By Theorem \ref{main 5} and the dual mixed volume inequality \eqref{dual mix vol ineq-log}, there is
    \[0=\tilde{V}_{\log}(\rR_0 f, \rR_0 f)\leq \tilde{V}_{\log}((\rR_0 f)^{\star}, \rR_0 f^{\star})\leq \frac{1}{n}\log\Big(\frac{|\rR_0 f^{\star}|}{|\rR_0 f|}\Big),\]
    with equality if and only if
    \[f(x)=f^{\star}(\phi^{-1}x-x_0),\]
    %,~~\text{and}~~\rR_0 f^{\star}=(\rR_0 f)^{\star},\]
    for some $\phi\in{\rm SL}(n)$ and $x_0\in\rn$.
\end{proof}

Before characterizing the equality cases of the affine logarithmic Sobolev inequality, we first illustrate how Theorem \ref{main 5} can be applied to clarify the equality conditions in Beckner’s logarithmic Sobolev inequality. To this end, we begin by recalling two results from Carlen and Loss \cite{CL}.

\begin{lem}\cite[Lemma 3]{CL}\label{cl lem 3}
    Let $f\in L^1(\rn)$ be a non-negative,  symmetric decreasing function and assume that for some $o\neq b\in\rn$, there exists some $c\in\rn$ and a  symmetric decreasing function $g$ such that
    \[IT_bf=T_cg,\]
    where
    \begin{equation}\label{trans}
        T_bf(x):=f(x-b),~~~~If(x)=|x|^{-2n}f\Big(\frac{x}{|x|^2}\Big).
    \end{equation}
    Then $f$ is uniformly bounded and $f(x)\leq C|x|^{-2n}$ for some constant $C$. Furthermore, the vector $c$ is given by
    \[c=b\int_{\rn}f(x)^{1+1/n}dx\bigg/\left(\int_{\rn}|x|^2f(x)^{1+1/n}+|b|^2\int_{\rn}f(x)^{1+1/n}dx\right).\]
\end{lem}

\begin{lem}\cite[Theorem 4]{CL}\label{cl thm 4}
    Let $f$ be as in Lemma \ref{cl lem 3} but assume in addition that
    \begin{equation}\label{cond 2}
    \int_{\rn}f(x)^{1+1/n}dx=\int_{\rn}|x|^2f(x)^{1+1/n}dx,
    \end{equation}
    and $b$ is chosen such that the angle $\phi$ defined by $\sin\phi=|b|/(1+|b|^2)^{1/2}$ is a non-rational multiple of $\pi$ if $n=1$ and $\phi$ is not an integer multiple of $\pi/4$ if $n\ge 2$. Then
    \[f(x)=C(1+|x|^2)^{-n}\]
    for some constant $C$.
\end{lem}

Building on the previous Lemma \ref{cl lem 3}, Lemma \ref{cl thm 4} and Theorem \ref{main 5}, we turn to the characterization of equality cases in Beckner's logarithmic Sobolev inequality.

\begin{thm}\label{cha}
For $\beta\in(0,\frac{1}{2})$, let { $f\in W^{\beta,2}(\rn)$} be a non-negative function such that $\|f\|_2=1$ and  
 $\int_{\rn}f(x)^2\log f(x)dx$ is finite. If
    \begin{align*}
    \gamma_0-\frac{2}{n}\int_{\rn}f(x)^2\log f(x)dx=\frac{1}{n\omega_n}\int_{\sn}\log\rho_{{\rm R}_0f}(u)du,
    \end{align*}
then
\[f(x)=a(1+\lambda |x-x_0|^2)^{-n/2},\]
for $x_0\in\rn$ and some $a,\lambda>0$.
\end{thm}

\begin{proof}
By the Riesz rearrangement theorem, we have
\begin{equation*}
    \di\frac{f(x)f(y)}{|x-y|^{n-\alpha}}dxdy\leq \di \frac{f^{\star}(x)f^{\star}(y)}{|x-y|^{n-\alpha}}dxdy,
\end{equation*}
and Theorem \ref{main 2} implies that
 \begin{align}\label{refined B-log}
 \gamma_0-\frac{2}{n}\int_{\rn}f(x)^2\log f(x)dx
\ge\frac{1}{n\omega_n}\int_{\sn}\log\rho_{\rR_0 f^{\star}}(u)du\ge\frac{1}{n\omega_n}\int_{\sn}\log\rho_{\rR_0 f}(u)du,
\end{align}
where the first term is invariant under Schwarz symmetrization. Note that
\[\frac{1}{n\omega_n}\int_{\sn}\log\rho_{\rR_0 f}(u)du=\tilde{V}_{\log}(B^n,\rR_0 f).\]
Then if equality holds in Beckner's logarithmic Sobolev inequality, by \eqref{refined B-log}, we have
\[\tilde{V}_{\log}(B^n,\rR_0 f)=\tilde{V}_{\log}(B^n,\rR_0 f^{\star}).\]
Theorem \ref{main 5} then implies that $f$ is a translate of $f^{\star}$.

Hence the optimizer of Beckner's logarithmic Sobolev inequality can only be a symmetric decreasing function up to a translation. If $f$ is an optimizer, by the conformal invariance of the logarithmic Sobolev inequality, 
\begin{equation*}
\begin{aligned}
    &T_bf(x):=f(x-b),\\
    &\tilde{I}f(x):=|x|^{-n}f(\frac{x}{|x|^2}),
\end{aligned}    
\end{equation*}
and $\tilde{I}T_bf$ are also optimizers, where $b\in\rn$ is arbitrary. Therefore there is a symmetric decreasing function $g$ and $c\in\rn$ such that
\begin{equation}\label{eq 2}
    \tilde{I}T_bf=T_c g.
\end{equation}

By the definition  of $T_b$, $I$ and $I_2$, we have $(\tilde{I}f)^2=I(f^2)$ and $(T_b f)^2=T_b(f^2)$. Hence \eqref{eq 2} yields
\[IT_b(f^2)=T_c(g^2),\]
where $f^2\in L^1(\rn)$ is a  symmetric decreasing function. As the preceding identity meets the hypotheses of Lemma \ref{cl lem 3}, it follows that $f^2$ is uniformly bounded and $f(x)^2\leq C|x|^{-2n}$. 

It remains to verify that $f^2$ fulfills condition \eqref{cond 2}. If not, set
\[f_{\lambda}^2(x)=f(\lambda x)^2,~~~~ \lambda=\left(\frac{\int_{\rn}|x|^2f(x)^{2+2/n}dx}{\int_{\rn}f(x)^{2+2/n}dx}\right)^{1/2},\]
where $\lambda$ is well-defined thanks to the uniform bound on $f$ and the decay $f(x)^2\leq C|x|^{-2n}$. Then function $f_{\lambda}(x)^2$ remains $L^1$-integrable, non-negative, symmetric decreasing, and it satisfies \eqref{cond 2}. 
Finally, by Lemma \ref{cl thm 4}, we obtain
\[f(x)=a(1+\lambda |x-x_0|^2)^{-n/2},\]
where $x_0\in\rn$ and $a,\lambda>0$ such that $\|f\|_2=1$.
\end{proof}

By Proposition \ref{eqiv}, Beckner's logarithmic Sobolev inequality can be written as
\begin{equation*}\label{B-log}
    \gamma_0-\frac{2}{n}\int_{\rn}f(x)^2\log f(x)dx
\ge\frac{1}{n\omega_n}\int_{\sn}\log\rho_{\rR_0 f}(u)du.
\end{equation*}
Consequently, Theorem \ref{cha}, together with the straightforward sufficiency given later, fully describes its equality cases.

Theorem \ref{cha} will in turn play a crucial role in deriving the equality cases for the affine logarithmic Sobolev inequality, which we now consider.

\begin{thm}[Characterization of optimizers]
    Let $\beta\in(0,\frac{1}{2})$. Suppose $f\in W^{\beta,2}(\rn)$ is non-negative and $\|f\|_2=1$. If
    \begin{equation}\label{eq1}
        \gamma_0-\frac{2}{n}\int_{\rn}f(x)^2\log f(x)dx=\frac{1}{n}\log\Big(\frac{|\rR_0f|}{\omega_n}\Big),
    \end{equation}
    then
    \[f(x)=a(1+|\phi(x-x_0)|^2)^{-\frac{n}{2}},\]
    for $x_0\in\rn$, $a>0$ and $\phi\in{\rm GL}(n)$ such that $\|f\|_2=1$.
\end{thm}

\begin{proof}
    By Corollary \ref{rri vol} and Theorem \ref{main 2}, we have
        \[\gamma_0-\frac{2}{n}\int_{\rn}f(x)^2\log f(x)dx\ge\frac{1}{n}\log\Big(\frac{|\rR_0f^{\star}|}{\omega_n}\Big)\ge\frac{1}{n}\log\Big(\frac{|\rR_0f|}{\omega_n}\Big).\]
    Then \eqref{eq1} implies $|\rR_0 f|=|\rR_0 f^{\star}|$ and thus $f$ is a translate of $f^{\star}\circ \phi^{-1}$ for some $\phi\in{\rm SL}(n)$.

    Note that for $f^{\star}$ satisfying \eqref{eq1}, we have
    \[\gamma_0-\frac{2}{n}\int_{\rn}f(x)^2\log f(x)dx=\frac{1}{n}\log\Big(\frac{|\rR_0f^{\star}|}{\omega_n}\Big)=\frac{1}{n\omega_n}\int_{\sn}\log\rho_{\rR_0 f^{\star}}(u)du.\]
    Hence $f^{\star}$ is an optimizer of Beckner's logarithmic Sobolev inequality, which is given by
    \[f^{\star}(x)=a(1+\lambda|x|^2)^{-\frac{n}{2}},\]
    for $a,\lambda>0$ such that $\|f^{\star}\|_2=1$. This concludes the proof.
\end{proof}
%provides the characterization of the equality case of the second in $|\rR_0 f|\leq |\rR_0f^{\star}|$.

% Then if the equality attains in the affine logarithmic Sobolev inequality, then
% \[\int_{\sn}\log\rho_{\rR_0 f^{\star}}(u)du=\int_{\sn}\log\rho_{\rR_0 f}(u)du,\]
% and the following result implies that in this case $f$ is a translate of $f^{\star}$.

For the sufficient condition of the equality cases, by applying the continuity of $\rho_{\Ra f}$ with respect to $\alpha$, we provide a simple proof. By the equality cases of the affine HLS inequalities for $\alpha\in(0,n)$, 
\[f_{\alpha}(x)=(1+|x|^2)^{-\frac{n+\alpha}{2}}\]
is an optimizer. %A natural idea is to check if $f_0$ will realize the equality case in Theorem \ref{main 2}. 
Then we have
\[\gamma_{n,\alpha}\|f_{\alpha}\|_{\frac{2n}{n+\alpha
}}^2=\frac{\|f_{\alpha}\|_2^2}{\alpha}n\omega_n^{\frac{n-\alpha}{n}}|\Ra f_{\alpha}|^{\frac{\alpha}{n}}.\]
Our approach is to establish the continuity of $|\Ra f_{\alpha}|$ at $\alpha=0$, and then we can differentiate both sides and attain the equality in \eqref{ineq 0}. 

%Our method is to make use of equality cases of HLS inequalities and certain preparations are necessary. 

%{\clr Let us recall the Lemma \ref{HL}. For a fixed $u\in\sn$, if we set
%\[g_0(t)=\frac{1}{\|f_0\|_2^2}\int_{\rn}f_0(x)f_0(x+tu)dx\]
%and $\phi(t)=-\log(g(t)/g(0))$, unfortunately $t\mapsto\phi(t)/t$ is not an increasing function on $(0,\infty)$. This indicates that it is unclear whether $\rho_{\rR_0 f_0}(u)$ is finite. Moreover, we do not know whether $\rho_{\Ra f_0}(u)$ is continuous with respect to $\alpha$. The real issue here is to address these problems. }

Since $f_{\alpha}$ is radially symmetric which implies that $\rho_{\Ra f_{\alpha}}(u)$ takes the same value for all $u\in\sn$. Then let $u\in\sn$ be fixed. We first have the following result from the equality case of the affine HLS inequality.

\begin{lem}\label{conti a}
    For {$\alpha\in(0,n)$} and $f_{\alpha}(x)=(1+|x|^2)^{-\frac{n+\alpha}{2}}$,
    \[\rho_{\Rf_{\alpha}}(u)^{\alpha}=\frac{\alpha}{\|f_{\alpha}\|_2^2}\int_{0}^{\infty}t^{\alpha-1}\int_{\rn}f_{\alpha}(x)f_{\alpha}(x+tu)dxdt,\]
    is continuous with respect to $\alpha$.
\end{lem}

\begin{proof}
    Using the equality cases in the affine HLS inequalities,
    \[\gamma_{n,\alpha}\|f_{\alpha}\|_{\frac{2n}{n+\alpha}}^2=n\omega_n^{\frac{n-\alpha}{n}}|\Sa_{\alpha}|^{\frac{\alpha}{n}}=\frac{\|f_{\alpha}\|_2^2}{\alpha}n^{1+\frac{\alpha}{n}}\omega_n\rho_{\Ra f_{\alpha}}(u)^{\alpha},\]
    and therefore
    \[\rho_{\Ra f_{\alpha}}(u)^{\alpha}=\frac{\alpha\gamma_{n,\alpha}}{\|f_{\alpha}\|_2^2\omega_nn^{1+\frac{\alpha}{n}}}\left(\int_{\rn}(1+|x|^2)^{-n}dx\right)^{\frac{n+\alpha}{n}}.\]

Note that, by polar coordinate, 
\begin{equation*}\label{l2}
\begin{aligned}
\|f_{\alpha}\|_2^2=n\omega_n\int_{0}^{\infty}(1+r^2)^{-(n+\alpha)}r^{n-1}dr=\frac{n\omega_n}{2}\int_0^1t^{\frac{n}{2}+\alpha}(1-t)^{\frac{n}{2}-1}dt
\end{aligned}
\end{equation*}
 is the beta function. Therefore $\|f_{\alpha}\|_2^2$ is continuous for $\alpha>-\frac{n}{2}$ and thus $\rho_{\Ra f_{\alpha}}(u)$ is continuous with respect to $\alpha>0$. 
\end{proof}

By the proof of Lemma \ref{conti for cs},  we can similarly establish the continuity of $\rho_{\Rf_{\alpha}}(u)$ at $\alpha=0$, which is
\begin{equation}\label{conti 0}
    \lim_{\alpha\rightarrow 0^+}\rho_{\Rf_{\alpha}}(u)=\rho_{\rR_0f_0}(u).
\end{equation}
Then we have the following result which concludes the proof of Theorem 2.

\begin{thm}\label{suff-thm}
    There is equality in the first inequality of \eqref{ineq 0} if $f(x)=a(1+\phi|x-x_0|^2)^{-\frac{n}{2}}$, where $x_0\in\rn$, $a>0$ and $\phi\in {\rm GL(n)}$ such that $\|f\|_2=1$. There is equality in the second inequality of \eqref{ineq 0} if $f$ is radially symmetric.   
\end{thm}

\begin{proof}    
    For the first inequality. Let $f_{\alpha}(x)=c\left(1+\lambda|x|^2\right)^{-\frac{n+\alpha}{2}}$, where $c, \lambda>0$ such that $\|f_0\|_2=1$. Since $f_{\alpha}$ is the optimal function of the affine HLS inequality, there is
    %\[\|f_{\alpha}\|^2_{\frac{2n}{n+\alpha}}=\left(\int_{\rn}(1+|x|^2)^{-n}dx\right)^{\frac{n+\alpha}{n}},\]
    
    %\[\|f_{\alpha}\|_2^2=\int_{\rn}(1+|x|^2)^{-(n+\alpha)}dx,\]
    %and
    \begin{equation*}\label{eqn}
    \alpha\gamma_{n,\alpha}\|f_{\alpha}\|^2_{\frac{2n}{n+\alpha}}=\|f_{\alpha}\|_2^2n\omega_n\Big(\frac{|\rR_{\alpha}f_{\alpha}|}{\omega_n}\Big)^{\frac{\alpha}{n}}.
    \end{equation*}

    By \eqref{conti 0} and the fact that $\Rf_{\alpha}$ is a ball, we can do the differentiation on both sides as $\alpha\rightarrow 0^+$, and obtain
    \[\omega_n\|f_0\|_2^2\log\Big(\frac{|\rR_0f_0|}{\omega_n}\Big)=\gamma_0n\omega_n\|f_0\|_2^2+\omega_n\|f_0\|_2^2\log\|f_0\|_2^2-2\omega_n\int_{\rn}f_0(x)^2\log f_0(x)dx,\]
    where $\gamma_0$ is given in Theorem \ref{main 2}. Since $\|f_0\|_2=1$, the equation we obtain is
    \[\frac{1}{n}\log\Big(\frac{|\rR_0 f_0|}{\omega_n}\Big)=\gamma_0-\frac{2}{n}\int_{\rn}f_0(x)^2\log f_0(x)dx.\]
    Since $f_0$ is radially symmetric, all the terms in inequality \eqref{ineq 0} coincide.

    For the second inequality, since it is directly from Jensen's inequality, equality holds if and only if $\rR_0f$ is a ball. Hence, equality holds in the second inequality if $f$ is radially symmetric.
\end{proof}
\vskip 12pt

\noindent{\bf Acknowledgements:} 
This research was funded in whole or in part by the Austrian Science Fund (FWF) doi/10.55776/37030. For open access purposes, the author has applied a CC BY public copyright license to any author accepted manuscript version arising from this submission.

\section*{Declaration}
The author declares that this manuscript is her original work, has not been published elsewhere, and is not under consideration by any other journal.

The author has no conflicts of interest to disclose. All research conducted in this work complies with the ethical standards of the author’s institution.

\bibliographystyle{abbrv}
%\bibliography{new_bib}

\end{document}